\def\paperversion{2}
\ifnum\paperversion=1
\documentclass[moor,blindrev]{informs4}
\fi
\ifnum\paperversion=2
\documentclass[moor]{informs4}
\fi

\usepackage{eqndefns-left} %
\RequirePackage{tgtermes}
\RequirePackage{newtxtext}
\RequirePackage{newtxmath}
\RequirePackage{bm}
\RequirePackage{endnotes}

\OnePointTwoSpacedXI

\usepackage{algorithm,algorithmicx,amsmath,amssymb}
\usepackage{algpseudocode}
\usepackage{bbm}
\usepackage[mathlines,pagewise,left]{lineno}
\usepackage{amssymb}
\usepackage{color,xcolor,tikz,tikz-3dplot}
\usepackage{booktabs}
\usetikzlibrary{arrows,shapes,chains}  
\usepackage{multirow, array}
\usepackage{caption}
\usepackage{xspace}
\usepackage{comment}
\usepackage[toc,page]{appendix}
\usepackage{graphicx} %
\usepackage{booktabs,stackengine} %
\usepackage{pdflscape}
\usepackage[export]{adjustbox}
\usepackage{mathtools}
\usepackage{hyperref}
\usepackage{bm}
\usepackage[normalem]{ulem}
\usepackage{enumerate}

\usepackage{cleveref}
\usepackage[subtle, indent=normal, mathspacing=normal]{savetrees}
\usepackage{multirow}
\usepackage{subcaption} %
\usepackage{todonotes,bm}
\usepackage{url}

\def\R{\mathbb{R}}

\def\I{\mathcal{I}}

\def\hat{\widehat}

\def \A{\mathcal{A}}
\newcommand{\set}[1]{\mathcal{#1}}

\def\O{{\mathcal O}}

\def\H{{\mathcal H}}

\def\X{{\mathcal X}}
\def\Z{{\mathcal Z}}

\def\Re{{\mathbb R}}

\def\Q{{\mathcal Q}}
\def\O{{\mathcal O}}

\newcommand{\exclude}[1]{}

\DeclareMathOperator{\supp}{supp}

\usepackage{algorithm}
\usepackage{algorithmicx}
\usepackage{algpseudocode}
\algdef{SE}[DOWHILE]{Do}{doWhile}{\algorithmicdo}[1]{\algorithmicwhile\ #1}

\DeclareMathOperator{\conv}{conv}
\DeclareMathOperator{\cl}{cl}

\DeclareMathOperator{\Int}{int}

\DeclareMathOperator{\ext}{ext}

\usepackage{tikz}
\usepackage{tikz-3dplot}
\usetikzlibrary{shapes,decorations,arrows,calc,arrows.meta,fit,positioning}
\tikzset{
	-Latex,auto,node distance =1 cm and 1 cm,semithick,
	state/.style ={ellipse, draw, minimum width = 0.7 cm},
	point/.style = {circle, draw, inner sep=0.04cm,fill,node contents={}},
	bidirected/.style={Latex-Latex,dashed},
	el/.style = {inner sep=2pt, align=left, sloped}
}

\newcommand{\Qaffine}{\ensuremath{\mathcal{Q}^\text{Aff}}}
\newcommand{\comonotone}{comonotone\xspace}
\newcommand{\innerProd}[2]{\left\langle#1,#2\right\rangle}
\newcommand{\np}{\textsf{NP}}

\newcommand{\permMap}{\ensuremath{\Psi}}
\newcommand{\defeq}{\ensuremath{\triangleq}}

\newenvironment{pf}{\proof{Proof.}}{\hfill\Halmos\endproof}

\usepackage[sort&compress]{natbib}
 \bibpunct[, ]{[}{]}{,}{n}{}{,}%
 \def\bibfont{\small}%

\EquationsNumberedThrough    %

\TheoremsNumberedThrough     %
\ECRepeatTheorems  %

\MANUSCRIPTNO{MOOR-0001-2024.00}

\begin{document}

\RUNTITLE{A Geometric Perspective on Polynomially Solvable Convex Maximization}

\TITLE{A Geometric Perspective on Polynomially Solvable \\ Convex Maximization}

\ARTICLEAUTHORS{%
\AUTHOR{Shaoning Han}
\AFF{Department of Mathematics, National University of Singapore, Singapore, \EMAIL{shaoninghan@nus.edu.sg}}

\AUTHOR{Liangju Li}
\AFF{Department of Mathematics, National University of Singapore, Singapore, \EMAIL{e1583542@u.nus.edu}}

\AUTHOR{Yongchun Li}
\AFF{School of Data Science, The Chinese University of Hong Kong, Shenzhen, China, \EMAIL{yongchunli@cuhk.edu.cn}}
} %

\ABSTRACT{
Convex maximization encompasses a broad class of optimization problems and is generally $\np$-hard, even for low-rank objectives. 
This paper investigates structural conditions under which convex maximization becomes polynomially solvable. 
From a geometric perspective, we introduce \emph{comonotonicity}, a  structural property of the feasible region crucial for problem tractability, and establish mathematical characterizations of this property.
Under comonotonicity and  mild additional assumptions, we develop a unified enumerative framework showing that fixed-rank convex maximization is polynomially solvable. 
This viewpoint recovers several known tractability results that previously required separate analyses, such as fixed-rank convex matroid maximization and sparse principal component analysis (SPCA).
Furthermore, for the more structured class of \emph{standard comonotone} feasible regions, we refine the analysis via a lifting technique to achieve a square-root improvement in the complexity bound.  
Finally, applications to SPCA and its variants illustrate the broad applicability and effectiveness of the proposed framework.
}

\KEYWORDS{Comonotone sets, Convex maximization, Polynomial complexity, Hyperplane arrangement, Sparse PCA, Global optimization, Lifting} 

\maketitle

\section{Introduction} \label{sec:intro} 
In this paper, we consider convex maximization problems of the form:
\begin{equation}\label{eq:general-concave}
	\begin{aligned}
		\max_{x\in\R^n}\;&f(A x)\\
		\text{s.t. }&x\in\set{X}\subseteq \Re^n,
	\end{aligned}
\end{equation}
where $A\in \Re^{r\times n}$ has full row rank with $r\le n$, the function $f: \Re^r \to \Re$ is convex, and $\set{X}\subseteq\R^n$ is a generic feasible region. We assume that the set of global optimizers of \eqref{eq:general-concave} is nonempty. Following the literature (e.g., \cite{goyal2013fptas,han2025compact}), the integer $r$ is  often referred to as the \textit{rank} of the objective and serves as a measure of its nonlinearity.

Convex maximization \eqref{eq:general-concave} has long been a central topic in global optimization. Many fundamental optimization problems can be formulated directly or reformulated equivalently as convex maximization problems \cite{benson1995concave}, including zero-one integer programming \cite{giannessi1976connections,raghavachari1969connections}, bilinear programming \cite{benson1985finite}, and difference-of-convex programming \cite[Chapter~3]{horst2013global}. The broad applicability of \eqref{eq:general-concave} in turn makes \eqref{eq:general-concave} extremely challenging to solve, both theoretically and computationally. Indeed, even under certain special settings, problem \eqref{eq:general-concave} remains \np-hard. For example, \citet{pardalos1991quadratic} prove that maximizing a convex quadratic function of rank two ($r=2$) over a polytope is \np-hard. Moreover, \citet{mittal2013fptas} show that \eqref{eq:general-concave} generally does not admit a constant-factor approximation algorithm unless $\textsf{P}=\np$.

Despite its general intractability, this paper aims to study \emph{when problem \eqref{eq:general-concave} becomes polynomially solvable}. For any vector $x\in\R^n$, we denote by $\supp(x)=\{i\in[n]:\,x_i\neq 0\}$ the support of $x$. We make the following assumption. 
\begin{assumption}\label{assume:supp}
	For any subset $S\subseteq[n]$, the following restriction of problem \eqref{eq:general-concave} can be solved in polynomial time $\textsf{T}_1$:
	\begin{align}\label{eq:assumption1}
		\max_{x\in\set{X}} \left\{ f(A x) : \supp(x) = S \right\}.  
	\end{align}  
\end{assumption}
In other words, \Cref{assume:supp} means that problem \eqref{eq:general-concave} becomes tractable once the support of decision variables is fixed. Note that the fixed-support set $\{x\in\set{X}:\supp(x)=S \}$ is generally nonclosed. Consequently,  problem~\eqref{eq:assumption1} may not always attain its optimum.
More precisely, \Cref{assume:supp} is understood in the following algorithmic sense: there exists a polynomial-time algorithm 
that either  returns an optimal solution when one exists, or certifies that no optimal solution exists. 

To illustrate the modeling scope of \Cref{assume:supp}, we present three representative settings in which it arises naturally. First, the assumption is automatic in binary settings where $\set{X}\subseteq\{0,1\}^n$. Indeed, for any subset $S\subseteq[n]$, fixing $\supp(x)=S$ uniquely determines the vector $x$ by setting $x_i=1$ for $i\in S$ and $x_i=0$ for $i\notin S$. Thus, in this case, \eqref{eq:assumption1} reduces to evaluating the objective $f(Ax)$.     Second, \Cref{assume:supp} naturally appears in modern data science problems where support selection itself is part of the decision, such as sparse principal component analysis (SPCA) \cite{hotelling1933analysis,jeffers1967two} and variable selection for two-sample tests (2ST) \cite{wang2023variable}. To improve the effectiveness and interpretability of the underlying statistical model, such problems typically impose a zero-``norm'' constraint to promote sparse solutions. Once the support is fixed, \eqref{eq:assumption1} amounts to selecting a subset of features or samples and often reduces to a more conventional and tractable optimization problem.   Third, \Cref{assume:supp} is also closely connected to active set or manifold identification in classical nonlinear programming \cite{hare2004identifying,oberlin2006active}. Suppose the feasible region $\set{X}$ is described by equality constraints together with nonnegativity constraints $x\ge 0$. This standard form is without loss of generality (WLOG), since general inequality constraints can be converted into it by introducing slack variables.  Then in this context, determining the active constraints is equivalent to determining the support of continuous variables $x$.  Once the active constraints at optimality are identified, problem \eqref{eq:general-concave} locally reduces to an equality-constrained problem and is often more tractable in principle.

When $\set{X}$ or $\conv(\set{X})$ is a polytope, where $\conv(\set{X})$ represents the convex hull of $\set{X}$, the convex maximization problem \eqref{eq:general-concave} always admits an optimal solution which is an extreme point of $\set{X}$. This fact motivates a naive solution strategy: enumerate all extreme points of $\set{X}$, evaluate the objective value $f(Ax)$ at each such point, and then select the best one. However, the brute-force approach can easily struggle when the number of extreme points of $\set{X}$ becomes prohibitively large. This  may happen even for structurally simple feasible sets; for example, $\set{X}=\{x\in\R^n:\,\sum_{i=1}^n x_i\le s,\,0\le x\le 1  \}$, where $s\le n$ is a positive integer. A more effective approach exploits the rank of the objective and instead enumerates the extreme points of the lower-dimensional image $A\set{X}\defeq\{Ax:\,x\in\set{X}\}.$ To illustrate the importance of the rank parameter $r$ for the  tractability of \eqref{eq:general-concave}, consider the case $r=1$. In this simple case, $\conv(A\set{X})$ is just a bounded interval with two endpoints. Accordingly, solving \eqref{eq:general-concave} boils down to evaluating $f$ at these two points.  When $r=2$, one can efficiently enumerate the extreme points of the two-dimensional set $A\set{X}$ using parametric linear programming \cite{hassin1989maximizing, atamturk2009submodular, yu2023strong}. However, the parametric linear programming approach does not extend to higher ranks, where the structure of $\set{X}$ must come into play. 

When $\set{X}$ encodes the bases of a matroid and $r$ is fixed, \citet{onn2003convex} proves that \eqref{eq:general-concave} is polynomially solvable. The author constructs a zonotope from the edges (one-dimensional faces) of $\conv(\set{X})$ and reduces problem~\eqref{eq:general-concave} to solving $\O(n^{2r-2})$ linear programs. With matroids as a prototypical example, \citet{onn2004convex} subsequently extend the zonotope-reduction technique to more general discrete settings. They show that, for any fixed rank $r$, \eqref{eq:general-concave} can be solved in time  polynomial in dimension $n$ and the number of edge directions of $\conv(\set{X})$, provided that these directions can be efficiently constructed and that one has access to an oracle for solving linear programs over $\set{X}$. In particular, for an integer linear set $\set{X}=\{ x\in\mathbb{Z}^n:\,Qx=b,\,\ell\le x\le u \}$, where $Q$, $b$, $\ell$ and $u$ are compatible integer matrices and vectors, the \emph{Graver basis} of $Q$ forms a set of all edge directions of $\conv(\set{X})$ \cite{de2009convex}. Several refinements and generalizations along this line also exist in the literature \cite{de2008n,deza2018primitive,deza2021vertices,onn2001vector,scott2025normal}. We refer readers to the monograph \cite{onn2010nonlinear} for a comprehensive treatment of this direction. However, when \eqref{eq:general-concave} involves continuous variables and $\conv(\set{X})$ is no longer polyhedral, $\conv(\set{X})$ typically has infinitely many edges and extreme points. It is unclear how to extend the above results to the nonpolyhedral setting.

When $\conv(\set{X})$ is nonpolyhedral,  some specialized applications of \eqref{eq:general-concave}  are known in the literature to be polynomially solvable, most notably SPCA and the least trimmed squares problem (LTS) \cite{rousseeuw1985multivariate,gomez2025outlier}. In both problems, the main challenge is to optimally select a  subset of size $k$  from $n$ candidate  features or samples \cite{hossjer1995exact,li2024beyond}. Once the subset is determined, SPCA and LTS reduce to the standard PCA and ordinary least-square estimation problems, respectively. Given an $n\times n$ covariance matrix of rank $r$, \citet{asteris2014sparse} show that single-component SPCA can be solved by computing $\O(n^r)$ standard PCA problems. More recently, \citet{del2023sparse} demonstrates that general SPCA can be solved in $\O(n^{\min\{d,r\}(r^2+r)})$ oracle time, where $d$ denotes the number of principal components. Furthermore, \cite{del2023sparse} also establishes the first polynomial-time complexity for disjoint SPCA on fixed rank matrices.  For LTS, \citet{hossjer1995exact} develops an exact $\O(n^2\log n)$ algorithm for single-feature instances. \citet{mount2014least} later develop a topological plane-sweeping algorithm that solves general LTS problems in $\O(n^{r+2})$ oracle time, where $r$ denotes the number of features in this context.   To the best of our knowledge, apart from these specialized results, convex maximization problems with nonpolyhedral $\conv(\set{X})$ have not been systematically investigated in literature.

\subsection{Contributions and outline}
Following the literature, we also treat the rank $r$ as a fixed constant in this paper. Below we summarize the main contributions of the paper and outline the rest of the paper. Our contributions are three-fold:
\vspace{0.5em}

\noindent\emph{1. We introduce the notion of {comonotonicity}---a set property that plays a fundamental role in the polynomial solvability of \eqref{eq:general-concave}, and establish basic characterizations and properties of (standard) comonotone sets.} 

Roughly speaking, a set $\set{X}$ is called comonotone if for each linear optimization problem over $\set{X}$, the ordering pattern of its optimal solution depends only on that of the linear objective coefficients. When the two orderings always coincide, the set is further called standard comonotone.  It turns out that comonotonicity can capture the symmetric structure of a set and brings several well-known ordering-related set classes under a common umbrella,  including in particular matroids \cite{whitney1935abstract,lovasz1980matroid, calinescu2011maximizing, white1992matroid} and permutation-invariant sets \cite{kalra2020learning,kim2022convexification,lee2019set,li2025partial}.  %
While this phenomenon of order dependence is prevalent in the literature, comonotonicity  has not yet been  recognized and studied as a set property in its own right. In this paper, we explicitly identify this structure and investigate its implications.
\vspace{1em}

\noindent\emph{2. We show that maximizing a fixed-rank convex function over comonotone sets can be solved in polynomial oracle time under Assumption~\ref{assume:supp}.} 

When problem~\ref{eq:general-concave} admits a comonotone feasible region, we develop a general theoretical framework for problem \eqref{eq:general-concave} that generates $\O(n^{2r})$ supports of decision variables,  one of which coincides with the support of an optimal solution.  Consequently, under \Cref{assume:supp}, \eqref{eq:general-concave} can be solved by addressing $\O(n^{2r})$ fixed-support subproblems of the form \eqref{eq:assumption1}. This framework applies to any comonotone feasible set and, in a unified manner, recovers the known polynomial solvability result of convex matroid maximization \cite{onn2003convex} and nonpolyhedral applications surveyed above.
\vspace{1em}

\noindent\emph{3. For standard comonotone sets, we further refine the analysis via a lifting approach and obtain a square-root improvement in oracle complexity.}
\nopagebreak

We work in a suitably lifted space of the parameters under consideration and refine the support-generation framework for standard comonotone sets. This reduces the number of candidate supports from $\O(n^{2r})$ to $\O(n^{r+1})$. As representative illustrations, Table~\ref{tab:application} summarizes the resulting complexity bounds when our approach is applied to PCA-related instances of \eqref{eq:general-concave}, where ``--" indicates that no prior polynomial solvability result is known in the literature. Moreover, compared with the existing complexity analyses \cite{asteris2014sparse,del2023sparse}, our approach bypasses intricate problem-specific combinatorial constructions and yields a substantially simpler derivation. 

\begin{table}[h]
	\centering
	\caption{Oracle complexity  for PCA-related application instances of \eqref{eq:general-concave} }\label{tab:application}
		\setlength{\tabcolsep}{3pt}\renewcommand{\arraystretch}{1.5}%
		\begin{tabular}{c|c|c|c|c|c}\hline
			& {Single SPCA} & Nonnegative {SPCA}& SPCA & {Disjoint SPCA} & {2ST}  \\
			\hline
			This paper & $\O(n^r)$ & $\O(n^r)$  & $\O\left(n^{(r^2+r)/2}\right)$ & $\O\left(\left( n(d+1)^2 \right)^{\frac{d(r^2+r+2)}{2}-1}\right)$ & $\O(n^{r+1})$ \\ \hline
			Literature \cite{asteris2014sparse, del2023sparse} & $\O(n^r)$ & $-$&$\O\left(n^{\min\{d,r\}(r^2+r)}\right)$ & $\O\left(n^{d^2(r^2+r)/2}\right)$ &$-$ \\ 
			\hline
		\end{tabular}
\end{table}

The remainder of the paper is organized as follows. In Section~\ref{sec:comonotonicity}, we establish the theoretical foundation of comonotone sets. In Section~\ref{sec:poly}, we develop a unified framework for maximizing convex functions over comonotone sets and obtain polynomial complexity results under mild conditions. In Section~\ref{sec:inv}, we derive improved complexity bounds under standard comonotonicity by means of lifting arguments.  We also show there how additional sign conditions can be incorporated to sharpen the bounds further. In Section~\ref{sec:extension}, we extend the framework to settings with quasi-convex objectives or feasible regions that are not exactly comonotone but remain sufficiently close to a comonotone set. Section~\ref{sec:app} is devoted to applications. Finally, in Section~\ref{sec:conclusion}, we conclude the paper.

\subsection{Notations}
For a positive integer $n$, we let $\Re^n$ and $\Re_+^n$ denote the set of all  $n$-dimensional vectors and nonnegative vectors, respectively. We let $\mathbb{Z}_+$ and $\mathbb{Z}_{++}$ be the set of all nonnegative integers and positive integers, respectively.
Let $[n]=\{1,2,\cdots,n\}$. We denote by $\Pi_n$ the set of all permutations over $[n]$. We let $e$ denote the vector of all ones,  let $e^i$ denote the $i$-th coordinate vector for each $i$, let $I$ denote the identity matrix, with their size being clear in context. For a vector $x\in \Re^n$, we let $\|x\|_0$  denote the number of nonzero entries of $x$,   and let $|x|=(|x_1|, \cdots, |x_n|)^{\top}\in\R^n$ contain the absolute entries of $x$.
For a matrix $X$ and a positive integer $d$, we let $\|X\|_F$ denote its Frobenius norm, let $\|X\|_0$ denote the number of nonzero rows of $X$, and let $\|X\|_{(d)}$ denote the sum of its $d$ largest eigenvalues.  
For a symmetric matrix $X$, we let $\lambda_{\max}(X)$ denote its largest eigenvalue, and given an index set $S\subseteq[n]$, let $X_{S,S}$ denote a principal submatrix of $X$ indexed by $S$. For a set $\set{D}\subseteq\R^n$, we denote by $\conv(\set{D})$ its convex hull,  by $\Int\set{D}$ its interior, by $\cl\set{D}$ its closure, and by $\dim(\set{D})$ its dimension.
Given a hyperplane $H=\{x\in \Re^n: a^{\top}x=b\}$, we denote its two open half-spaces by 
$$H^{>}=\{x\in \Re^n: a^{\top}x>b\}, \text{ and } H^{<}=\{x\in \Re^n: a^{\top}x<b\}.$$

For a permutation $\pi\in \Pi_n$ and a vector $x\in \Re^n$, we say that $x$ is \emph{sorted} by $\pi$ if $x_{\pi(1)}\ge \cdots \ge x_{\pi(n)}$ holds. We also call $\pi$ an ordering of $x$.   Note that the permutation that sorts a given vector $x$ is not unique  when there are ties. For example, the vector $x=[1, 0,1]^{\top}$ can be sorted by  $\pi=(1, 3, 2)$ or $(3, 1, 2)$. We also denote by $x(\pi)\in\R^n$ the vector obtained by permuting the entries of $x$ according to $\pi$, that is, $[x(\pi)]_{\pi(1)}\ge [x(\pi)]_{\pi(2)}\ge\dots\ge [x(\pi)]_{\pi(n)}$. For example, given the vector $x=(-1,0,1)^\top$ and the permutation $\pi=(2,1,3)$, the permuted vector $x(\pi)=(0,1,-1)$.
Given a permutation $\pi$, we define a cone
$\Z(\pi) = \{x \in \Re^n:  x_{\pi(1)} \ge \cdots \ge x_{\pi(n)} \}$ and its nonnegative counterpart $\Z^+(\pi) =\Z(\pi)\cap\R_+^n$. For example, when $n=2$ and $\pi=(2,1)$, one has $\Z(\pi)=\{x\in\R^2:\,x_2\ge x_1 \}$. Naturally, we have $x(\pi)\in\Z(\pi)$ for all $x\in\R^n$ and $\pi\in\Pi_n$.

\section{Comonotone sets}\label{sec:comonotonicity}
In this section,  we first formally introduce the notion of comonotonicity. We then establish structural properties and basic characterizations of comonotone sets, with particular emphasis on \emph{standard comonotonicity}. We also present a range of examples to illustrate the definition and to anchor the geometric intuition behind the results. These structural results and examples serve as the geometric cornerstone of our complexity analysis.

We now state the definition of (standard) comonotone sets, which captures the geometric intuition that optimal solutions to a linear program should inherit the sorting pattern of the cost vector. This alignment is formalized through a permutation mapping that specifies the induced ordering by the objective.
\begin{definition}[Comonotone sets]\label{def:map}
	A set $\X\subseteq \Re^n$ is called \emph{\comonotone} if there exists a \emph{permutation mapping} $\permMap:\Pi_n\to \Pi_n$ such that for every permutation $\pi \in \Pi_n$ and every cost vector $v\in \Z(\pi)$,  
	whenever the problem 
	$$\max_{x\in \X}v^\top x,$$
	attains its optimum, it admits an optimal solution $x^*$ sorted by $\permMap(\pi)$, that is,
	$$x^*\in \Z\left(\permMap(\pi)\right).$$ If one can choose $\permMap$ to be the identity mapping, namely $\permMap(\pi)=\pi$ for all $\pi\in\Pi_n$,  then $\X$ is called
	\emph{standard \comonotone}.%
\end{definition} 
Below provides a simple example to illustrate the  concept.
\begin{example}\label{eg:simple}
	The set $\X=\{x\in \{0,1\}^n: e^{\top} x=s\}$ with $s\in [n]$ is standard \comonotone. Indeed, if $v\in \Z(\pi)$, an optimal solution to $\max\{ v^\top x:\;x\in\set{X} \}$ is obtained by setting
	$x_{\pi(1)}=\cdots=x_{\pi(s)}=1$ and $x_{\pi(s+1)}=\cdots=x_{\pi(n)}=0$, which
	belongs to $\Z(\pi)$. Thus, one can take $\permMap(\pi)=\pi$ for any $ \pi\in\Pi_n$. \hfill\Halmos
\end{example}

While comonotonicity is defined through  an admissible permutation mapping $\permMap$, the mapping is best viewed as a \emph{witness} of an underlying order compatibility built into $\set{X}$, rather than a genuine source of freedom. 
Theorem~\ref{thm:standard-comonotone} below highlights the resulting combinatorial rigidity: when $\permMap$ is surjective, so that no ordering type is missing or collapsed, there is essentially only one substantive realization, that is, $\set{X}$ must be a standard comonotone set. We first recall the rearrangement inequality to prove Theorem~\ref{thm:standard-comonotone}.
\begin{lemma}[Rearrangement inequality {\cite[p. 261]{hardy1952inequalities}}]\label{lem:rearrangement-ineq}
	For all $x,v\in\R^n$, it holds %
	\[ [v(\pi)]^\top x(\pi)\ge v^\top x,\quad\forall \pi\in\Pi_n. \] 
\end{lemma}

For a permutation mapping $\permMap:\Pi_n\to\Pi_n$ and $k\in\mathbb{Z}_+$, let $\permMap^k$ denote the $k$-fold composition of $\permMap$ (i.e., $\permMap^k=\underbrace{\permMap\circ\cdots\circ\permMap}_{k\text{ times}}$), where $\permMap^0$ is understood as the identity mapping by convention.

\begin{theorem}\label{thm:standard-comonotone}
	Suppose that $\set{X}\subseteq\R^n$ is a compact \comonotone set under permutation mapping $\permMap:\Pi_n\to\Pi_n$. If $\permMap$ is surjective, then $\set{X}$ is standard \comonotone.
\end{theorem} 
\begin{pf}
	Note that for a \comonotone set, the choice of permutation mapping may not be unique in Definition~\ref{def:map}. Although $\permMap$ itself is not necessarily identity, our goal is to show that the
	identity mapping can be used instead in Definition~\ref{def:map}. This is equivalent to prove that for any $\pi\in\Pi_n$ and $v\in\set{Z}(\pi)$, one can find a solution $x^*\in\set{Z}(\pi)\cap\argmax\{ v^\top x:\;x\in\set{X}\}$. 
	
	Consider an arbitrary but fixed $\pi\in\Pi_n$ and $v\in\mathcal Z(\pi)$. 
	We first show that there exists $\hat{k}\in\mathbb Z_{+}$ such that 
	$\permMap^{\hat{k}}(\pi)=\pi$. Since $\Pi_n$ is finite, the iterates
	\[
	\pi,\ \permMap(\pi),\ \permMap^2(\pi),\ \ldots
	\]
	can take only finitely many values. Hence, there exist indices $0\le a<b$ satisfying
	$\permMap^{a}(\pi)=\permMap^{b}(\pi)$. Because $\permMap:\Pi_n\to\Pi_n$ is
	surjective and $\Pi_n$ is finite, $\permMap$ must be bijective, which implies that $\permMap^a$ is also bijective and thus injective. 
	Applying the injectivity of $\permMap^a$ to the equality $\permMap^{a}(\pi)=\permMap^{b}(\pi)=\permMap^a(\permMap^{b-a}(\pi))$
	yields $\pi=\permMap^{b-a}(\pi)$. Setting $\hat{k}\defeq b-a$ gives
	$\permMap^{\hat{k}}(\pi)=\pi$, as claimed.
	
	Next, we propagate optimality through the iterates of $\permMap$ and carry it back to $\mathcal Z(\pi)$. Specifically, we define a sequence of permuted cost vectors by $v^k=v(\permMap^{k-1}(\pi))$ for $k\in\mathbb{Z}_{++}$. It follows that $v^1=v$ and by construction $v^k\in\Z(\permMap^{k-1}(\pi))$ for all $k$. Moreover, since $\set{X}$ is \comonotone under $\permMap$, for each $k\in\mathbb{Z}_{++}$, one can choose an optimal solution $x^k$ to $\max_{x\in\set{X}}(v^k)^\top x$ such that $x^k$ is sorted by $\permMap(\permMap^{k-1}(\pi))=\permMap^k(\pi)$. The optimality of $x^k$ implies that 
	\begin{equation}\label{eq:optimality-chain-ineq}
		\innerProd{v^k}{x^k}\ge \innerProd{v^k}{x^{k-1}}, \quad\forall k. 
	\end{equation}
	Furthermore, since $v^k$ is a permutation of $v^{k-1}$, and both $v^k$ and $x^{k-1}$ are sorted by the same permutation $\permMap^{k-1}(\pi)$, one can deduce from \Cref{lem:rearrangement-ineq} that 
	\begin{equation}\label{eq:rearrangement-chain-ineq}
		\innerProd{v^k}{x^{k-1}}\ge \innerProd{v^{k-1}}{x^{k-1}},\quad \forall k.
	\end{equation}
	Chaining inequalities \eqref{eq:optimality-chain-ineq} and \eqref{eq:rearrangement-chain-ineq} and combining $\permMap^{\hat{k}}(\pi)=\pi$, one obtains
	\[ \innerProd{v^1}{x^1}\le\innerProd{v^2}{x^1}\le\innerProd{v^2}{x^2}\le\dots\le \innerProd{v^{\hat{k}+1}}{x^{\hat{k}}}=\innerProd{v\left(\permMap^{\hat{k}}(\pi)\right)}{x^{\hat{k}}}=\innerProd{v(\pi)}{x^{\hat{k}}}. \] 
	Because $v^1=v$ and $v(\pi)=v$ (due to $v\in\Z(\pi)$), the inequality chain reduces to 
	\begin{equation}\label{eq:reduced-chain}
		\innerProd{v}{x^1}\le\innerProd{v}{x^{\hat{k}}}.
	\end{equation} Moreover, since $x^1$ is optimal for the linear coefficients $v$, the inequality~\eqref{eq:reduced-chain} implies that $x^{\hat{k}}$ is also an optimal solution to $\max\{ v^\top x:\;x\in\set{X} \}$. Finally, by construction $x^{\hat{k}}$ is sorted by
	$\permMap^{\hat{k}}(\pi)=\pi$, i.e., $x^{\hat{k}}\in \Z(\pi)$. Hence we find the desired solution $x^*=x^{\hat{k}}$,
	which completes the proof.   
\end{pf}

To illustrate Theorem~\ref{thm:standard-comonotone}, we specialize to the planar case $n=2$. 
Since  $\Pi_2=\{(1, 2), (2,1)\}$ consists of only two permutations, if $\permMap$ is not surjective, then the compact comonotone set $\set{X}$ must lie entirely in one ordering cone $\Z(1,2)$ or $\Z(2,1)$.
In other words, comonotonicity is always {standard} in $\R^2$ unless the  set $\set{X}$ forces a fixed ordering; see Figure~\ref{fig:2D-comonotone-sets} for illustration. 
\begin{figure}[h]
	\centering
	\begin{subfigure}[t]{0.32\textwidth}
		\centering
		\begin{tikzpicture}[scale=1.,>=Stealth,baseline]
			\draw[->] (0,0) -- (3,0) node[below] {$x_1$};
			\draw[->] (0,0) -- (0,3) node[left] {$x_2$};
			\draw[dashed,black,-] (0,0) -- (2.7,2.7);
			\fill[black!15, opacity=0.7] (0.7,1.9) .. controls (1.2,2.6) and (2.4,2.3) .. (2.3,1.3)
			.. controls (1.2,1.4) and (1.2,0.4) .. (0.7,0.7) -- cycle;
			\draw[black!60] (0.7,1.9) .. controls (1.2,2.6) and (2.4,2.3) .. (2.3,1.3)
			.. controls (1.2,1.4) and (1.2,0.4) .. (0.7,0.7) -- cycle;
			\node[black] at (1.15,1.55) {\scriptsize $\set{X}$};
			\draw[->,black] (2,1.35) -- +(1.,0.) node[below] {\scriptsize $v_1\!\ge\! v_2$};
			\fill[black!70] (2.3,1.35) circle (3pt);
			\draw[->,black] (1.55,1.95) -- +(0.3,0.9) node[left] {\scriptsize $v_2\!\ge\! v_1$};
			\fill[black!70] (1.65,2.25) circle (3pt);
		\end{tikzpicture}
		\caption{$\set{X}$ is standard comonotone}
	\end{subfigure}\hfill
	\begin{subfigure}[t]{0.32\textwidth}
		\centering
		\begin{tikzpicture}[scale=1,>=Stealth,baseline]
			\draw[->] (0,0) -- (3,0) node[below] {$x_1$};
			\draw[->] (0,0) -- (0,3) node[left] {$x_2$};
			\draw[dashed,gray,-] (0,0) -- (2.7,2.7);
			
			\def\myshape{plot [smooth cycle, tension=0.8] coordinates {
					(1.0, 0.6)   %
					(2, 1.4)   %
					(2.6, 1.0)   %
					(2.2, 0.5)   %
					(1.4, 0.5)   %
			}}
			
			\fill[black!15, opacity=0.7] \myshape;
			\draw[black!60] \myshape;
			
			\node[black] at (1.75, 0.85) {\scriptsize $\mathcal{X}$};
			
		\end{tikzpicture}

		\caption{$\set{X}\subseteq\set{Z}(1,2)$}
	\end{subfigure}\hfill
	\begin{subfigure}[t]{0.32\textwidth}
		\centering
		
		\begin{tikzpicture}[scale=1,>=Stealth,baseline]
			
			\draw[->] (0,0) -- (3,0) node[below] {$x_1$};
			\draw[->] (0,0) -- (0,3) node[left] {$x_2$};
			
			\draw[dashed,gray,-] (0,0) -- (2.7,2.7)
			node[above, black] {\scriptsize $x_1=x_2$};
			\begin{scope}[shift={(-0.3,-0.3)}, scale=0.85]
				\def\myshape{plot [smooth cycle, tension=0.7] coordinates {
						(0.6, 1.2)
						(1.1, 2.5)
						(2.1, 2.6)
						(1.5, 1.7)
				}}
				
				\fill[black!15, opacity=0.7] \myshape;
				\draw[black!60] \myshape;
				
				\node[black] at (1.35, 2.0) {\scriptsize $\mathcal{X}$};
				
			\end{scope}
			
		\end{tikzpicture}
		\caption{$\set{X}\subseteq\set{Z}(2,1)$}
	\end{subfigure}
	\caption{Examples of comonotone sets in $\R^2$}\label{fig:2D-comonotone-sets} 
\end{figure}
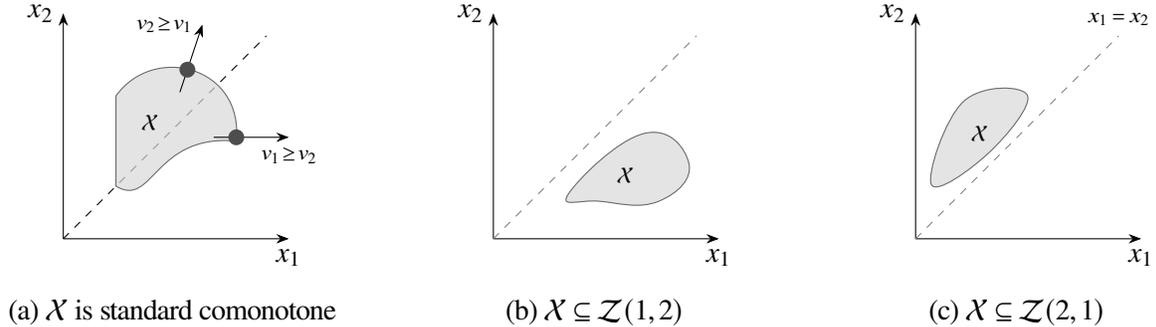 

\subsection{Properties of standard \comonotone sets}
In this subsection, we study properties of standard \comonotone sets.  
We begin with the planar case to build geometric intuition. While restricted to $n=2$, it provides a clear lens for understanding the restrictions imposed by standard comonotonicity, or equivalently, the structural information it carries. As shown below, standard comonotonicity in $\R^2$ can be characterized using only two directions, akin to checking extreme rays in polyhedral theory.
\begin{proposition}\label{prop:2D-finite-num-LP}
	Let $\set{X}\subseteq\R^2$ be compact. Then $\set{X}$ is  standard \comonotone  if and only if, for all $v\in\{e,-e\}$ and $\pi\in\Pi_2=\{ (1,2), (2,1) \}$,  the problem $\max\{ v^\top x:\, x\in\set{X}\}$ admits an optimal solution $\bar x\in\Z(\pi)$.
\end{proposition}
\begin{pf}
	The necessity  follows directly from \Cref{def:map} and the fact that $\{e,-e\}\subseteq\Z(1,2)\cap\Z(2,1)$.
	To prove sufficiency, we may assume without loss of generality
	that $v_1> v_2$ (the case $v_1<v_2$ is symmetric after swapping the two coordinates). It remains to prove that there exists an optimal solution  $\bar x\in\argmax\{v^\top x:\,x\in\set{X}  \}$ such that $\bar x_1\ge\bar x_2$. 
	
	Take any $\tilde x\in\argmax\{ v^\top x: \,x\in\set{X} \}$. Then if $\tilde x_1\ge \tilde x_2$,  we are done. Otherwise, assume $\tilde x_1<\tilde x_2$ and let $\bar v=\frac{v_1+v_2}{2}e$.  By the given precondition, $\bar v\in\Z(1,2)$ implies that there exists $\bar x\in\argmax\{ \bar v^\top x:\, x\in\set{X} \}$ with $\bar x_1\ge\bar x_2$. Next, we show the chain of inequalities 
	\begin{equation}\label{eq:2D-finite-num-LP}
		v^\top\tilde{x}\le \bar v^\top\tilde x\le \bar v^\top\bar x\le v^\top\bar x.
	\end{equation} 
	Indeed, the first inequality follows from
	\[ v^\top\tilde{x}- \bar v^\top\tilde x=\frac{1}{2}(v_1-v_2)(\tilde x_1-\tilde x_2)\le0, \]
	the second from $\bar x$ is optimal for the linear objective $\bar v$, and the third from
	\[ \bar v^\top\bar x - v^\top\bar x =-\frac{1}{2}(v_1-v_2)(\bar x_1-\bar x_2)\le 0. \]
	Because $\tilde x$ is maximal for the linear objective $v$, \eqref{eq:2D-finite-num-LP} implies that $\bar x$ is also an optimal solution for $v$. The conclusion follows from $\bar x_1\ge\bar x_2$. 
\end{pf}

We next give an example illustrating Proposition~\ref{prop:2D-finite-num-LP}.
\begin{example} Consider the sets shown in Figure~\ref{fig:standard-2D-comonotone-sets}. For the linear objective $x_1+x_2$, $x^1$ and $x^2$ are maximizers over these sets, where $x^1 \in \mathcal{Z}(2,1)$ and $x^2 \in \mathcal{Z}(1,2)$; $x^3$ and $x^4$ are minimizers over theses sets, where $x^3\in\set{Z}(2,1)$ and $x^4\in\set{Z}(1,2)$. By \Cref{prop:2D-finite-num-LP}, all the three sets are  standard comonotone. In fact, any set $\mathcal{X}$ satisfying $\{x^1,x^2,x^3,x^4\}\subseteq \mathcal{X}$ and contained in the strip delimited by the two dotted lines must be standard comonotone. \hfill\Halmos 
	\begin{figure}[h]
		\centering
		\begin{subfigure}[t]{0.32\textwidth}
			\centering
			\begin{tikzpicture}[scale=1.,>=Stealth,baseline]
				\draw[->] (0,0) -- (3,0) node[below] {$x_1$};
				\draw[->] (0,0) -- (0,3) node[left] {$x_2$};
				\draw[dashed,gray,-] (0,0) -- (2.7,2.7) node[above right] { };
				
				\draw[thick, dotted,gray,-] (3,1.) -- (1.,3) node[above right] {};
				\draw[thick, dotted,gray,-] (0,1.5) -- (1.5,0) node[above right] {};
				\fill[black!15, opacity=0.7] (1.5,2.5) .. controls (3,1) and (1,0.5) .. (3,1)
				.. controls (2.5,0.4) and (1.0,0.3) .. (1,0.5).. controls (2.,2) and (1.0,0.6) .. (0.3,1.2) -- cycle;
				\draw[black!60] (1.5,2.5) .. controls (3,1) and (1,0.5) .. (3,1)
				.. controls (2.5,0.4) and (1.0,0.3) .. (1,0.5).. controls (2.,2) and (1.0,0.6) .. (0.3,1.2) -- cycle;
				\node[black!70] at (1.15,1.55) {\scriptsize $\set{X}_1$};
				\fill[black] (1.5,2.5) circle (2pt) node [above] {$x^1$};
				\fill[black] (3,1) circle (2pt) node [above] {$x^2$};
				\fill[black] (0.3,1.2) circle (2pt) node [above] {$x^3$};
				\fill[black] (1.,0.5) circle (2pt) node [below] {$x^4$};
			\end{tikzpicture}
		\end{subfigure}\hfill
		\begin{subfigure}[t]{0.32\textwidth}			
			\centering
			\begin{tikzpicture}[scale=1.,>=Stealth,baseline]
				\draw[->] (0,0) -- (3,0) node[below] {$x_1$};
				\draw[->] (0,0) -- (0,3) node[left] {$x_2$};
				\draw[dashed,gray,-] (0,0) -- (2.7,2.7) node[above] {};
				\draw[thick, dotted,gray,-] (3,1.) -- (1.,3) node[above right] {};
				\draw[thick, dotted,gray,-] (0,1.5) -- (1.5,0) node[above right] {};
				
				\fill[black!15, opacity=0.7] (1.5,2.5) -- (3,1)
				.. controls (2.5,0.4) and (1.0,0.3) .. (1,0.5)-- (0.3,1.2) -- cycle;
				\draw[black!60] (1.5,2.5) -- (3,1)
				.. controls (2.5,0.4) and (1.0,0.3) .. (1,0.5)-- (0.3,1.2) -- cycle;
				\node[black!70] at (1.15,1.55) {\scriptsize $\set{X}_2$};
				\fill[black] (1.5,2.5) circle (2pt) node [above] {$x^1$};
				\fill[black] (3,1) circle (2pt) node [above] {$x^2$};
				\fill[black] (0.3,1.2) circle (2pt) node [above] {$x^3$};
				\fill[black] (1.,0.5) circle (2pt) node [below] {$x^4$};
			\end{tikzpicture}
		\end{subfigure}\hfill
		\begin{subfigure}[t]{0.32\textwidth}			
			\centering
			\begin{tikzpicture}[scale=1.,>=Stealth,baseline]
				\draw[->] (0,0) -- (3,0) node[below] {$x_1$};
				\draw[->] (0,0) -- (0,3) node[left] {$x_2$};
				\draw[dashed,gray,-] (0,0) -- (2.7,2.7) node[above, black] {\scriptsize $x_1=x_2$};
				
				\draw[thick, dotted,gray,-] (3,1.) -- (1.,3) node[above right] {};
				\draw[thick, dotted,gray,-] (0,1.5) -- (1.5,0) node[above right] {};
				\fill[black!15, opacity=0.7] (1.,3) -- (1.5,2.5) .. controls (1.5,1.5) and (1.5,1.5) .. (2.5,1.5) -- (3,1)
				.. controls (1.,1.) and (1.4,0.2) .. (1.2,0.3) -- (1,0.5).. controls (1.2,1.2) and (1.2,1.2) .. (0.5,1) -- (0.3,1.2) .. controls (0.2,1.4) and (1.,1.) .. cycle;
				\draw[black!60] (1.,3) -- (1.5,2.5) .. controls (1.5,1.5) and (1.5,1.5) .. (2.5,1.5) -- (3,1)
				.. controls (1.,1.) and (1.4,0.2) .. (1.2,0.3) -- (1,0.5).. controls (1.2,1.2) and (1.2,1.2) .. (0.5,1) -- (0.3,1.2) .. controls (0.2,1.4) and (1.,1.) .. cycle;
				\node[black!70] at (1.15,1.55) {\scriptsize $\set{X}_3$};
				\fill[black] (1.5,2.5) circle (2pt) node [above] {$x^1$};
				\fill[black] (3,1) circle (2pt) node [above] {$x^2$};
				\fill[black] (0.3,1.2) circle (2pt) node [above] {$x^3$};
				\fill[black] (1.,0.5) circle (2pt) node [below] {$x^4$};
			\end{tikzpicture}
		\end{subfigure}
		\caption{Examples of standard comonotone sets in $\R^2$}\label{fig:standard-2D-comonotone-sets}
	\end{figure}
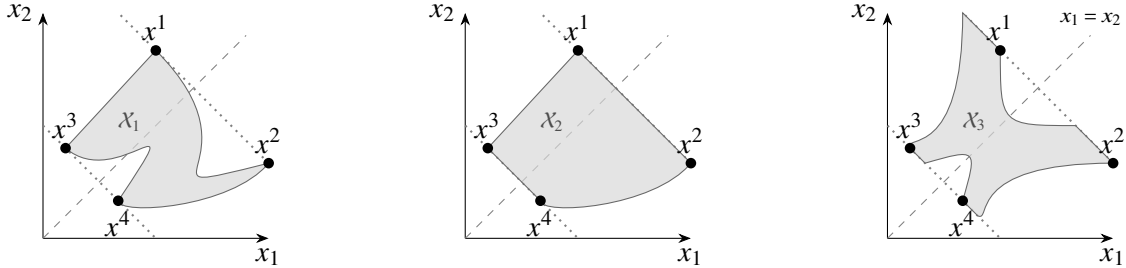
\end{example}

In $\mathbb{R}^2$, ordering information is encoded by only two
directions ($x_1+x_2$ or $-x_1-x_2$) as indicated by \Cref{prop:2D-finite-num-LP}. Unfortunately, the simplification does not directly extend to high dimensions. To see this, consider the distributed lattice $\set{X}=\{x\in\{0,1\}^3:\;x_1\ge x_2\ge x_3\}$. On the one hand, $0\in\argmax\{-e^\top x:\;x\in\set{X}  \}$ and $e\in\argmax\{e^\top x:\;x\in\set{X}  \}$, and moreover $\{0,e\}\subseteq\Z(\pi)$ for all $\pi\in\Pi_n$. On the other hand, for $v=(-2,3,-1)$, the corresponding linear maximization problem has the unique solution $x=(1,1,0)$, whose coordinate order is not compatible with that of $v$.  

The geometric intuition from the planar case facilitates the study of standard comonotonicity in higher dimensions. We first introduce a technical lemma for later use. 

\begin{lemma}\label{lem:unique-dense}
	Let $\X\subseteq\R^n$ be compact. Then for any $\pi\in\Pi_n$ and $v\in\Z(\pi)$, there exists
	a sequence $\{v^k\}\subseteq\Int\Z(\pi)$ such that (i)~$\lim\limits_{k\to\infty}v^k=v$ and (ii)~for every $k$, the problem $\max\left\{(v^k)^\top x:\ x\in\X\right\}$ admits a unique optimal solution.
\end{lemma}

\begin{pf}
	Let $\delta^*(u|\set{X}) \defeq \max\left\{u^\top x:\,x\in\set{X}\right\}$ be the \emph{support function} of $\X$, and $\set{D}$ be the set of points where $\delta^*(\cdot|\set{X})$ is differentiable. By \cite[Theorem~25.1]{rockafellar1970convex}, $\delta^*(\cdot|\set{X})$ is differentiable at $u$ if and only if the subdifferential $\partial\delta^*(u|\set{X})$ is a singleton. Moreover,  according to \cite[Corollary 23.5.3]{rockafellar1970convex}, we get
	\[\partial\delta^*(u|\set{X}) = \cl\conv\left(\argmax\left\{ u^\top x:\,{x \in \X}\right\}\right).\]
	Because a set is a singleton if and only if its closed convex hull is a singleton, it follows that that $\set{D}$ is exactly the set of points $u$ where $\{\max u^\top x:\, x\in\set{X}\}$  admits a unique optimal solution.
	
	Since $\X$ is compact, $\delta^*(\cdot|\set{X})$ is finite-valued and convex on $\R^n$. Hence, by \cite[Theorem 25.5]{rockafellar1970convex}, the function $\delta^*(u|\set{X})$ is  differentiable almost everywhere, that is, the complement of $\set{D}$ is a set of measure zero. In particular, $\Int\Z(\pi)\setminus\set{D}$ has also measure zero. This implies $\cl\left(\Int\Z(\pi)\cap\set{D}\right)=\cl\left(\Int\Z(\pi)\right)=\Z(\pi)$. Consequently, for any $v\in\Z(\pi)$, there exists a sequence $\{v^k\}\subseteq\Int\Z(\pi)\cap\set{D}$ such that $\lim\limits_{k\to\infty}v^k=v$. The conclusion follows.
\end{pf}

\Cref{thm:comonotone-symmetry} below characterizes standard comonotonicity in $\R^n$ and generalizes Proposition~\ref{prop:2D-finite-num-LP}. 

\begin{theorem}\label{thm:comonotone-symmetry}
	Given a compact set $\set{X}\subseteq\R^n$, the following statements are equivalent:
	\begin{enumerate}[(i)]
		\item The set $\set{X}$ is standard comonotone.
		\item For all $v\in\R^n$ and all $i,j\in[n]$, there exists a certain  $x^*\in\argmax\{ v^\top x:x\in\set{X} \}$ such that
		\begin{equation}\label{eq:comonotone-relation}
			(v_i-v_j)(x^*_i-x^*_j)\ge0.
		\end{equation}
		\item For all $v\in\R^n$, relation~\eqref{eq:comonotone-relation} holds for all $x^*\in\argmax\{ v^\top x:x\in\set{X} \}$ and all $i,j\in[n]$. 
		\item For all $v\in\R^n$ and all $i,j\in[n]$ with $v_i=v_j$, there exist $\bar x,\tilde{x}\in\argmax\{ v^\top x:x\in\set{X} \}$ such that $\bar x_i\ge\bar x_j$ and $\tilde{x}_i\le\tilde{x}_j$.
	\end{enumerate}
\end{theorem}
\begin{pf} We prove the equivalence following the cycle $(i)\Rightarrow (iv)\Rightarrow (iii)\Rightarrow (ii)\Rightarrow (i)$.
	
	$(i)\Rightarrow(iv)$:
	Since $v_i=v_j$, there exist two permutations $\bar{\pi}$ and $\tilde{\pi}$ such that $v\in\set{Z}(\bar{\pi})\cap\set{Z}(\tilde{\pi})$, where the index $i$ precedes $j$ in $\bar{\pi}$, and $j$ precedes $i$ in $\tilde{\pi}$. From the definition of standard comonotonicity, there exist optimal solutions $\bar{x}\in\set{Z}(\bar{\pi})$ and $\tilde{x}\in\set{Z}(\tilde{\pi})$ to the problem $\max_{x\in\set{X}}v^\top x$. It follows that $\bar x_i\ge\bar x_j$ and $\tilde{x}_i\le\tilde{x}_j$.
	
	$(iv) \Rightarrow (iii)$:
	For any $v \in \R^n$, let $x^* \in \argmax\{ v^\top x:x\in\set{X} \}$ and take arbitrary but fixed $i,j\in[n]$. Because relation~\eqref{eq:comonotone-relation} holds trivially when $v_i = v_j$, we assume WLOG that $v_i > v_j$. We then construct a new objective vector $\bar{v} \in \R^n$ by 
	\begin{equation*}
		\bar{v}_k = 
		\begin{cases} 
			\dfrac{v_i+v_j}{2}, & \text{if } k \in \{i,j\}, \\ 
			v_k, & \text{otherwise}. 
		\end{cases}
	\end{equation*}
	Since $\bar{v}_i = \bar{v}_j$, condition $(iv)$ implies that there exists a $\bar{x} \in \argmax\{\bar{v}^\top x: x\in\set{X}\}$ such that $\bar{x}_i \ge \bar{x}_j$.
	
	By optimality, we have $v^\top x^* \ge v^\top \bar{x}$ and $\bar{v}^\top \bar{x} \ge \bar{v}^\top x^*$. Summing these two inequalities gives $(v - \bar{v})^\top x^* \ge (v - \bar{v})^\top \bar{x}, $
	which rearranges to 
	\begin{align*}
		0\le& (v - \bar{v})^\top (x^* - \bar{x}) 
		= \frac{v_i-v_j}{2} (x^*_i - \bar{x}_i) - \frac{v_i-v_j}{2} (x^*_j - \bar{x}_j)
		=\frac{1}{2} (v_i-v_j)\left[(x^*_i - x^*_j) - (\bar{x}_i - \bar{x}_j) \right].
	\end{align*}
	Since $v_i > v_j$, we further obtain $(x^*_i - x^*_j)-(\bar{x}_i - \bar{x}_j)\ge0$.
	Together with $\bar{x}_i - \bar{x}_j\ge0$, we have that $x^*_i - x^*_j \ge 0$. Consequently, $(v_i-v_j)(x^*_i-x^*_j) \ge 0$, which proves \eqref{eq:comonotone-relation}.
	
	$(iii)\Rightarrow(ii)$: This is immediate.
	
	$(ii)\Rightarrow(i)$: For any fixed $\pi\in\Pi_n$ and $v\in\set{Z}(\pi)$, we take the sequence $\{v^k \}$ in \Cref{lem:unique-dense}. Then for any $v^k$, $v^k\in\Int\Z(\pi)$ implies $v^k_{\pi(i)}-v^k_{\pi(j)}>0$ for all $i<j$. Moreover, since the problem $\max\left\{ (v^k)^\top x:\,x\in\set{X} \right\}$ has a unique maximizer, which we denote by $x^k$, one can deduce from (ii) that $x^k_{\pi(i)}\ge x^k_{\pi(j)}$ for all $i<j$. Namely, one has $x^k\in\set{Z}(\pi)$ for all $k$. Because $\set{X}$ is a compact set, the sequence $\{x^k\}$ has a clustering point. WLOG, we assume $\lim\limits_{k\to\infty} x^k=x^*$; otherwise we apply the arguments to its convergent subsequence. Because $\Z(\pi)$ is closed, it follows that $x^*\in\Z(\pi)$. Moreover, from the optimality of each $x^k$ we have   
	\[(v^k)^\top x^k\geq (v^k)^\top x,\qquad \forall x\in\set{X}.\] Taking the limit $k\to \infty$ on both sides leads to
	\[v^\top x^*\geq v^\top x,\qquad \forall x\in\set{X},\]
	implying $x^*\in\argmax\{v^\top x:\,x\in\set{X}\}$. This proves (i) and completes the proof.
\end{pf}

Theorem~\ref{thm:comonotone-symmetry} decomposes global comonotonicity into local pairwise ordering interactions. The complementary roles of parts~(ii) and (iv) are worth noting. Since \eqref{eq:comonotone-relation} holds trivially when $v_i=v_j$, part~(ii) captures only the strict comparison $v_i>v_j$, ensuring that the directional preference in $v$ is faithfully transmitted to the primal solution $x$. In contrast, part~(iv) does \emph{not} impose any strict ordering condition over $v$. It dictates that if $i$ and $j$ are equivalent in the dual space of cost vectors ($v_i=v_j$), then they must remain structurally indistinguishable in the primal space of $x$ as well. This requirement embodies a \emph{symmetry via conjugacy}, which becomes even more transparent for convex standard comonotone sets as we discuss next.

Since standard comonotonicity is defined through linear maximization, it is natural to ask whether it is preserved under convexification. Proposition~\ref{prop:cvx-hull} answers this question in the affirmative. 
\begin{proposition}\label{prop:cvx-hull}
	Let $\set{X}\subseteq\R^n$ be compact. Then $\set{X}$ is standard comonotone if and only if $\conv(\set{X})$ is standard comonotone.
\end{proposition}
\begin{pf}
	Assume first that $\set{X}$ is standard comonotone. Because 
	\begin{equation}\label{eq:sol-inclusion}
		\emptyset\neq\argmax\left\{ v^\top x:\,x\in\set{X}\right\}\subseteq \argmax\left\{ v^\top x:\,x\in\conv(\set{X})\right\},\qquad\forall v\in\R^n,
	\end{equation}
	it follows immediately from \Cref{def:map} that $\conv(\set{X})$ is standard comonotone.
	
	Conversely, suppose $\conv(\set{X})$ is standard comonotone. For any given $\pi\in\Pi_n$ and $v\in\Z(\pi)$, let $\{v^k\}$ be the sequence in \Cref{lem:unique-dense} converging to $v$. Because each cost vector $v^k$ admits a unique maximizer $x^k$ over $\conv(\set{X})$, such a solution $x^k$ must also be optimal over $\set{X}$ by \eqref{eq:sol-inclusion}. Applying the same limiting arguments as in the proof of \Cref{thm:comonotone-symmetry}, one can obtain a solution $x^*\in\Z(\pi)\cap\argmax\{ v^\top x:\,x\in\set{X} \}$. In turn, we deduce from \eqref{eq:sol-inclusion} that  $x^*\in\Z(\pi)\cap\argmax\{ v^\top x:\,x\in\conv(\set{X}) \}$. This proves the standard comonotonicity of $\conv(\set{X})$.
\end{pf}

Figure~\ref{fig:standard-2D-comonotone-sets} illustrates Proposition~\ref{prop:cvx-hull}, where $\set{X}_2=\conv(\set{X}_1)$ and both sets are standard comonotone. Building on the characterization in Theorem~\ref{thm:comonotone-symmetry}, we now specialize to the convex setting.
\begin{theorem}\label{thm:comonotone-convex-symmetry}
	Given a compact convex set $\set{X}$, the following statements are equivalent.
	\begin{enumerate}[(i)]
		\item The set $\set{X}$ is standard \comonotone.
		\item For all $v\in\R^n$ and all $i,j\in[n]$ with $v_i=v_j$, there exists a certain $x^*\in\argmax\{v^\top x:\;x\in\set{X}\}$ such that $x^*_i=x^*_j$.
		\item For all $v\in\R^n$, there exists a certain $x^*\in\argmax\{ v^\top x:\; x\in\set{X} \}$ such that $v_i=v_j$ implies $x_i^*=x_j^*$ for all $i,j\in[n]$.
	\end{enumerate}
\end{theorem}

\begin{pf}
	The implication $(iii)\Rightarrow(ii)$ is trivial. The direction $(ii)\Rightarrow (i)$ follows directly from \Cref{thm:comonotone-symmetry}. To prove the equivalence, it remains to prove $(i)\Rightarrow(iii)$.
	
	Suppose $\set{X}$ is standard comonotone. For any $v\in\R^n$, define $p(v)$ as the number of distinct pairs $(v_i,v_j)$ of entries of $v$, namely, 
	\[ p(v)\defeq\sum_{i=1}^n\sum_{j:j>i}\mathbb{I}\{v_i\neq v_j\}, \]
	where $\mathbb{I}\{v_i\neq v_j\}=1$ if $v_i\neq v_j$ and $0$ otherwise. We prove $(iii)$ by induction on $p(v)$. 
	In the base case where  $p(v)=n(n-1)/2$, one has $v_i\neq v_j$ for all $i\neq j\in[n]$ and thus, $(iii)$ holds trivially for such $v$.
	
	Fix $q\in\mathbb{Z}_{++}$ and assume $(iii)$ holds for all $v\in\R^n$ satisfying $n(n-1)/2\ge p(v)\ge q$. We aim to show that $(iii)$ also holds for every $v$ with $p(v)= q-1$. If  no such $v$ exists, then the claim follows. Otherwise, we consider any fixed $v$ with $p(v)=q-1$, and choose  $\alpha\subseteq[n]$ such that $v_i=v_j$ for all distinct $i, j\in\alpha$ while $v_i\neq v_j$ for all $i\in\alpha$ and $j\notin\alpha$. 
	
	For each $i\in\alpha$, define a sequence $v^{i,k}=v+\frac{1}{k}e^i$ for $k\in\mathbb{Z}_{++}$. Then it follows that $p({v}^{i,k})\ge q$ for all sufficiently large $k\in\mathbb{Z}_{++}$. By the induction hypothesis, for each $i\in\alpha$ and sufficiently large $k$, there exists  ${x}^{i,k}\in\argmax\{ ({v}^{i,k})^\top x:\,x\in\set{X} \}$ such that $(iii)$ holds true for the pair $(v^{i,k},x^{i,k})$. Particularly, since ${v}^{i,k}_{i'}=v_{i'}$ for all $i'\in[n]\backslash\{i\}$, we have for sufficiently large $k$ that \[
	v_{i'}=v_{j'}\ \Longrightarrow\ x^{i,k}_{i'}=x^{i,k}_{j'}\qquad
	\forall\, i\in\alpha,\ \forall\, i',j'\in[n]\setminus\{i\}.
	\]  Moreover, since $\set{X}$ is compact and $\lim\limits_{k\to\infty}v^{i,k}=v$ for all $i\in\alpha$, we may assume WLOG (passing to a convergent subsequence if necessary) that 
	\[ \lim_{k\to\infty}{x}^{i,k}=\bar{x}^i,\qquad \forall i\in\alpha.\]
	Then by construction, the limits $\{\bar{x}^i\}_{i\in\alpha}$ preserve the following ordering properties:
	\begin{enumerate}[(a)]
		\item $\bar{x}^{i}\in\argmax\left\{ v^\top x:\,x\in\set{X} \right\}$ for all $i\in\alpha$.
		\item  $v_{i'}=v_{j'}$ implies $\bar{x}^i_{i'}=\bar{x}^i_{j'}$ for all $i\in\alpha$ and $i',j'\in[n]\backslash\{i\}$.
		\item $\bar{x}^i_i\ge \bar{x}^i_{i'}$ for all $i\in\alpha$ and $i'\in\alpha\backslash\{i\}$. Indeed, comonotonicity of $\set{X}$ and $v^{i,k}_i> v^{i,k}_{i'}$ imply $\bar{x}^{i,k}_i\ge \bar{x}^{i,k}_{i'}$ by \Cref{thm:comonotone-symmetry} for each $k$. Then it follows by letting $k\to\infty$.
	\end{enumerate}
	
	Finally, we construct an optimal solution $x^*$ for $v$ from $\{ \bar{x}^i\}_{i\in\alpha}$ and verify that it satisfies $(iii)$.
	By property (b)-(c), let us first rewrite each $\bar{x}^i=\tilde{x}^i+t_i e^i$, where $t_i=\bar{x}_i^i-\bar{x}^i_{i'}\ge 0$ for a fixed $i'\in\alpha\backslash\{i\}$. Also, by construction, each $\tilde{x}^i$ satisfies the equalities:
	\begin{equation}\label{eq:implication-property}
		x_{i'}=x_{j'},\qquad\forall i',j'\in[n] \text{ s.t. }v_{i'}=v_{j'}.
	\end{equation}  
	If there exists $i\in\alpha$ such that $t_i=0$, then $\bar{x}^i=\tilde{x}^i$ satisfies $(iii)$ by property~(a) and \eqref{eq:implication-property}. This completes the inductive step. Hence, we may assume $t_i>0$ for all $i\in\alpha$ and define 
	\[ \lambda_i = \dfrac{\prod\limits_{j\in\alpha\backslash\{i\}}t_j}{\sum\limits_{i'\in\alpha}\prod\limits_{j\in\alpha\backslash\{i'\}}t_j}, \qquad\forall i\in\alpha.\]
	Note that $\sum_{i\in\alpha}\lambda_i=1$ and $\lambda>0$. Then we set
	\[x^*=\sum_{i\in\alpha}\lambda_i\bar{x}^i=\sum_{i\in\alpha}\lambda_i\tilde{x}^i+\dfrac{\prod\limits_{j\in\alpha}t_j}{\sum\limits_{i'\in\alpha}\prod\limits_{j\in\alpha\backslash\{i'\}}t_j}\left(\sum_{i\in\alpha}e^i \right).\]
	Since $\set{X}$ is convex, the set $\argmax\{ v^\top x:\,x\in\set{X} \}$ is also convex. Thus by property~(a), the convex combination above yields $x^*\in\argmax\{ v^\top x:\,x\in\set{X} \}$ . In addition,  each $\tilde{x}^i$ and the vector $\sum_{i\in\alpha}e^i$ satisfy \eqref{eq:implication-property}, and therefore so does their linear combination $x^*$. This shows that $(iii)$ holds for $v$ at $x^*$, completing the induction and the proof.
\end{pf}

Theorem~3 makes the symmetry viewpoint via conjugacy from Theorem~\ref{thm:comonotone-symmetry} particularly sharp in the convex setting: standard comonotonicity is fully characterized by the implication $v_i=v_j\Rightarrow x^*_i=x^*_j\;\forall i,j\in[n]$.

An extreme subclass of standard comonotone sets comprises permutation-invariant sets. Specifically, following \cite{kim2022convexification}, a set $\set{X}\subseteq\R^n$ is called \emph{permutation-invariant} if $x\in\set{X}$ implies $x(\pi)\in\set{X}$ for every $\pi\in\Pi_n$. Such sets naturally satisfy the symmetry requirement in Theorem~\ref{thm:comonotone-symmetry}(iv). Moreover, permutation invariance yields standard comonotonicity directly from Definition~\ref{def:map} and Lemma~\ref{lem:rearrangement-ineq}, since any attained maximizer can be permuted within $\set{X}$ to match the ordering of $v$ without decreasing the objective value. Figure~\ref{fig:standard-2D-comonotone-sets} depicts a permutation-invariant set $\set{X}_3$.
\begin{proposition}\label{prop:invariant}
	Every permutation-invariant set $\X$ is standard \comonotone.
\end{proposition}

We use Example~\ref{ex:ellipsoid} below to illustrate the results established in this section, which shows that for centered ellipsoids, comonotonicity and permutation invariance are equivalent.
\begin{example}\label{ex:ellipsoid}
	Let $\set{X}=\{ x:x^\top Q x\le 1 \}$, where $Q\in\R^{n\times n}$ is symmetric positive definite. We aim to show that if $\set{X}$ is \comonotone, then it is permutation-invariant. 
	
	Indeed, because $0$ lies in the interior of $\set{X}$,  one has $\Int\set{Z}(\pi)\cap\set{X}\neq\emptyset$ for each $\pi\in\Pi_n$.  Consequently, if $\set{X}$ is \comonotone under a mapping $\permMap:\Pi_n\to\Pi_n$, then $\permMap$ must be surjective. 
	By \Cref{thm:standard-comonotone}, it follows that $\set{X}$ is standard comonotone.   Now consider any $v\in\R^n$ and $\bar x\in\argmax\{ v^\top x:\, x\in\set{X} \}$.  Direct calculation yields that $\bar x=\frac{1}{\lambda} Q^{-1} v$ or equivalently $v=\lambda Q\bar x$, where $\lambda = \sqrt{v^\top Q^{-1} v}$. Combining above and invoking \Cref{thm:comonotone-convex-symmetry}~(ii), $\set X$ is \comonotone if and only if 
	\[ v_i=v_j\Rightarrow (Q^{-1}v)_i=(Q^{-1}v)_j,\quad\forall i,j\in[n],\;\forall v\in\R^n. \]
	Expanding this condition implies that for all $v\in\R^n$ and $i,j\in[n]$ with $v_i=v_j=t$,
	\[ [(Q^{-1})_{ii}+(Q^{-1})_{ij}]t+\hspace{-1em}\sum_{k\in[n]\setminus\{i.j\}}\hspace{-0.7em}(Q^{-1})_{ik}v_k= [(Q^{-1})_{ij}+(Q^{-1})_{jj}]t+\hspace{-1em}\sum_{k\in[n]\setminus\{i,j\}}\hspace{-0.7em}(Q^{-1})_{jk}v_k.\]
	Since $t$ and $\{v_k\}_{k\neq i,j}$ are arbitrary, we must have $(Q^{-1})_{ii}=(Q^{-1})_{jj}$, $(Q^{-1})_{ik}=Q^{-1}_{jk}$ for all $i,j\in[n]$ and $k\in[n]\setminus \{i,j\}$.  Thus all diagonal entries of $Q^{-1}$ are equal and all off-diagonal entries of $Q^{-1}$ are equal. 
	Equivalently, $Q^{-1}$ (and hence $Q$) is invariant under coordinate permutations, which implies that $\set{X}$ is permutation-invariant.\hfill\Halmos
\end{example}

\subsection{Non-standard \comonotone sets and matroids}
Moving beyond standard \comonotone sets, general comonotonicity can capture combinatorial structures that lack full symmetry. 
An important source of non-standard comonotone sets is  structure closely related to greedy algorithms.   Matroids provide a canonical instance of this kind.
Let $\set{M}=([n], \mathcal{B})$ be a matroid over $[n]$ with  a collection of bases $\mathcal{B}\subseteq 2^{[n]}$ \cite{whitney1935abstract}. For a set $S\subseteq[n]$, its incidence vector $x\in \{0,1\}^n$ is defined as $x_i=1$ if $i\in S$ and $0$ otherwise for all $i\in [n]$. In this paper, we identify $\set{M}$ with the corresponding set of incidence vectors $\X \subseteq \{0,1\}^n$. We also treat checking whether $S\in\set{B}$ or $S\in\set{M}$ as a constant-time operation.
\begin{proposition}\label{prop:matroid-comonotone}
	\label{lem:greedy}
	If $\X\subseteq\{0,1\}^n$ encodes a matroid or the set of bases of a matroid, then $\set{X}$ is \comonotone. Moreover, the associated permutation mapping $\permMap$ can be  accessed in $\O(n\log n)$ time.
\end{proposition}

\begin{pf}
	Let $\set{M}=([n], \mathcal{B})$ be a matroid. We first assume $\mathcal{X}$ represents the bases $\mathcal{B}$. Consider any $\pi\in\Pi_n$ and any vector $v\in \set{Z}(\pi)$.
	For the problem $\max\{v^\top x:\, x\in\set{X} \}$, the best-in greedy algorithm scans elements in the order $\pi(1),\ldots,\pi(n)$ and
	selects an element whenever independence is preserved, i.e., it constructs a
	vector $x^*\in\{0,1\}^n$ by the recursion
	\begin{equation}\label{eq:best-in-greedy}
		x_{\pi(i)}^* = 1 \text{ if and only if } [x_{\pi(1)}^*, \cdots, x_{\pi(i-1)}^*, 1, 0, \cdots, 0 ]^{\top} \text{ represents an independence set in $\set{M}$}. 
	\end{equation}
	By \cite{edmonds1971matroids}, this procedure returns an optimal solution to
	$\max\{v^\top x:\,x\in\X\}$. 
	Importantly, the output $x^*$ depends solely on the feasibility check  of $\mathcal{M}$ and the permutation $\pi$, not on the specific values of $v$. Therefore, for all $v$ sorted by $\pi$, the greedy algorithm produces the same optimal solution $x^*$. This proves that a matroid is \comonotone. To be more specific, given the optimal solution $x^*$ for the natural order $\pi=(1,2,\dots,n)$, one can define $\sigma=\permMap(\pi)$ by requiring $\sigma(i)<\sigma(j)$ if either (1)~$x^*_{\sigma(i)}=1$ and $x^*_{\sigma(j)}=0$, or (2)~$x^*_{\sigma(i)}=x^*_{\sigma(j)}$ and $i<j$. Analogously, $\permMap(\pi)$ can be determined for an arbitrary permutation $\pi$ other than the natural one.
	
	We next consider the case where $\set X$ represents the matroid $\set{M}$, i.e., the incidence
	vectors of independent sets. Let $x^*$ be the vector given by \eqref{eq:best-in-greedy}. Then the solution obtained by setting $\bar x_i=x^*_i$ if $v_i>0$ and $x_i=0$ otherwise is optimal for $\max_{x\in\set{X}} v^\top x$ \cite{edmonds1971matroids}. Because $x^*$ and $\bar x$ are binary, it is clear that $x^*$ and $\bar x$ share the same order according to above modification. This implies that $\set X$ is \comonotone under $\permMap$.  In both cases, the time complexity is $\mathcal{O}(n \log n)$, dominated by the sorting step.
\end{pf}

The connection between greedy solvability and set structure has been extensively studied beyond matroids in literature. Some notable binary generalizations include {antimatroids} \cite{edelman1985theory}, and greedoids with strong exchange axiom and matroid embedding structures \cite{korte1984greedoids, helman1993exact}. In the continuous setting, Edmonds’ classical best-in greedy
algorithm extends to {polymatroids} and {base
	polyhedra} of submodular functions \cite{edmonds1970}.  In discrete convex analysis,  {M-convex sets} serve as the integral analogue of these structures \cite[Chapter~4]{murota2003}. One can readily verify the comonotonicity of above greedy-related families using the same reasoning as in \Cref{lem:greedy}; we do not repeat the details here. 

While comonotonicity exhibits rich structure and merits further study in its own right, we believe the current level of treatment is sufficient for this paper's purpose of understanding the tractability of \eqref{eq:general-concave}.

\section{Complexity analysis: A unified theoretical framework}\label{sec:poly}
In this section, we assume that the feasible set $\X$ of \eqref{eq:general-concave} is \comonotone. Leveraging this set property, we propose a general analysis framework that identifies a polynomial number of support candidates and guarantees that \eqref{eq:general-concave} has an optimal solution supported on one of them. 
Under \Cref{assume:supp}, \eqref{eq:general-concave} becomes tractable once the support is fixed; hence, enumerating the candidates and solving the corresponding subproblems yield a polynomial-time algorithm. Notably, our analysis applies to  an arbitrary \comonotone feasible region in \eqref{eq:general-concave}.

The proposed framework is developed in three steps as summarized in \Cref{fig1} and outlined below:
\begin{enumerate}[Step 1.]
	\item Leveraging the convexity of $f$, we show that \eqref{eq:general-concave} can be reduced to its \textit{linear counterpart} \eqref{eq:linear}, which maximizes a linear function $c^\top Ax$ over the original feasible set $\X$ for some $c\in \Re^r$. However, since $c$ relies on the optimal solution to \eqref{eq:general-concave} and  is not known a priori, we cannot solve \eqref{eq:linear} directly. 
	\item To circumvent this, we partition the parameter space  $\Re^r$ of $c$ into a polynomial number of regions based on optimality conditions of \eqref{eq:linear}. These conditions are derived from the comonotonicity structure of $\X$, and they ensure that all cost vectors within the same region share the same optimality certificate for the linear counterpart \eqref{eq:linear}.
	
	\item Finally, we show that each region created in Step~2 gives rise to a polynomial number of candidate supports. Enumerating them over all regions produces a polynomial-size list that is guaranteed to contain an optimal support for \eqref{eq:general-concave}.
\end{enumerate}
\begin{figure}[htb]
	\centering
	\begin{tikzpicture}[ every text node part/.style={align=center}]
		\node[state, rectangle, text width=5.8cm,minimum height=1cm] (l) at (-1.2,2) {1. Reduction to the linear counterpart \\(\Cref{prop:linear})};
		
		\node[state, rectangle, text width=2.8cm,minimum height=1cm] (c) at (4.5,2) {2. Space partition\\(\Cref{lem:HA})};
		
		\node[state, rectangle, text width=3.6cm,minimum height=1cm] (p) at (9.1,2) {3. Support construction
			(\Cref{them:supp})};
		\path (l) edge[line width=1.1] (c);
		\path (c) edge[line width=1.1] (p);
	\end{tikzpicture}
	\caption{A theoretical framework of complexity analysis.}
	\label{fig1}
\end{figure}
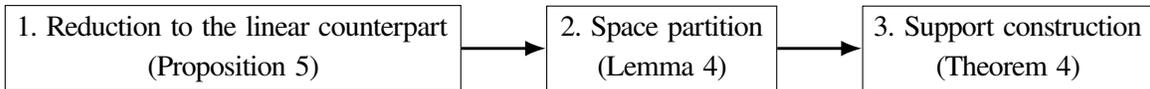

Beyond its broad applicability, the proposed framework is also highly adaptable. When additional problem structure is available (e.g., standard comonotonicity), each step can be tailored accordingly to tighten the complexity bounds.  To avoid redundancy, we do not repeat the full framework for special families of $\X$ in the subsequent sections; instead, we highlight only the necessary changes.

\subsection{Reduction to a linear program}
We first establish the reduction from \eqref{eq:general-concave} to its linear counterpart.
\begin{proposition}\label{prop:linear}
	For any feasible set $\X$,   there exists a vector $c\in \Re^r$ such that
	\begin{align}\label{eq:linear}
		\max_{x\in \X} c^{\top} A x
	\end{align}
	admits an optimal solution, and every optimal solution to the linear counterpart \eqref{eq:linear} is also optimal for the original problem \eqref{eq:general-concave}.
\end{proposition}
\begin{pf}
	Let $\hat x\in \X$ be an optimal solution to  problem \eqref{eq:general-concave}. Pick any subgradient $\hat c\in \Re^r$ of $f(\cdot)$ at  $A\hat x$. 
	We first show that $\hat x$ is also optimal for \eqref{eq:linear} with the cost vector $\hat c$. Suppose for contradiction that there exists a solution $x^*\in \X$ such that 
	$$\hat c^{\top}Ax^* > \hat c^{\top}A \hat x.$$ 
	By the subgradient inequality of convex $f$,  we have
	\[f(Ax^*)\ge f(A\hat x) +\hat c^{\top}(Ax^* - A \hat x)>f(A\hat x), \]
	which contradicts the optimality of $\hat x$ in \eqref{eq:general-concave}. Hence, $\hat x$ must be optimal for \eqref{eq:linear} with $\hat c$.
	
	Next,  we show that any optimal solution of \eqref{eq:linear} with the vector $\hat c$ is also optimal for \eqref{eq:general-concave}. Let $\tilde x\in \X$ be an arbitrary optimal solution of \eqref{eq:linear} with $\hat c$, and thus $\hat c^{\top}A \tilde x = \hat c^{\top}A \hat x$ holds. Then, applying the subgradient inequality of $f$ again yields 
	\[f(A\tilde x)\ge f(A\hat x) +\hat c^{\top}(A \tilde x - A \hat x)=f(A\hat x), \]
	where the inequality must hold with equality because \eqref{eq:general-concave} attains the optimum at  $\hat x$. It follows that $\tilde x$ also attains the optimal value of \eqref{eq:general-concave}. We thus complete the proof.
\end{pf}

It is worth noting that \Cref{prop:linear} is stronger than the classical results in the literature (see \cite{bector1970some,bereanu1972quasi} and \cite[Theorem~32.3]{rockafellar1970convex}). Roughly speaking, the latter assert that $\argmax\{ c^\top Ax:x\in\set{X} \}\cap\argmax\{f(A^\top x):\,x\in\set{X}\}\neq\emptyset$ for a certain $c\in\R^r$. Consequently, even if such a $c$ is known, one still needs to identify, among the maximizers of the linear counterpart \eqref{eq:linear}, a  solution that also maximizes $f(A^\top x).$ To avoid such an essentially bilevel program, the literature often explicitly  or implicitly (\cite[e.g.,][]{klouda2015exact, scott2025normal, onn2003convex, del2023sparse}) postulates a nondegeneracy assumption to ensure that the linear counterpart has a unique solution.
In contrast, because \Cref{prop:linear} establishes the stronger inclusion $\argmax\{ c^\top Ax:x\in\set{X} \}\subseteq\argmax\{f(A^\top x):\,x\in\set{X}\}$, any optimizer of \eqref{eq:linear} is automatically optimal for \eqref{eq:general-concave}. This frees us from above recurring technical nuisance and removes the need for such nondegeneracy assumptions.

\subsection{Space partition}
Since the cost vector $c$ in \Cref{prop:linear} is typically unknown, we cannot directly work with the linear counterpart \eqref{eq:linear}. 
Fortunately, under comonotonicity of $\set{X}$, we can instead rely on the following optimality condition for \eqref{eq:linear}: for any $\pi\in\Pi_n$ and $c\in \Re^r$,
\begin{equation}\label{eq:optimality-general-comonotonicity}
	A^\top c\in\Z(\pi) \Rightarrow \argmax\{A^\top c:\,x\in\set{X}\}\cap \Z(\permMap(\pi))\neq\emptyset,
\end{equation}
provided that $\argmax\{A^\top c:\,x\in\set{X}\}\neq\emptyset$.  This condition naturally leads to a partition of the parameter space of $c\in\R^r$, as we explain next.

We first introduce a technique from combinatorial geometry-- \textit{hyperplane arrangement}. Hyperplane arrangement is a standard tool for space partition (e.g., \cite{del2023sparse,onn2003convex,onn2004convex}).
Specifically, let $\H=\{H_i\}_{i\in [p]}$ be a collection of $p$ hyperplanes in $\Re^q$, where each $H_i$ denotes a hyperplane. The family $\H$ partitions the space $\Re^q$ into a collection of \textit{cells} with dimensions ranging from $0$ to $q$. Each cell $R$ is a possibly nonclosed region of the form $R=\cap_{i\in [p]} \tilde{H}_i$, where $\tilde H_i\in \{H_i, H_i^{>}, H_i^{<}\}$ represents either the hyperplane $H_i$ itself or one of its two open half-spaces. In particular, we call $R$ a \textit{$\tau$-cell} if $\dim(R)\le \tau$. The collection of all such cells is called the \textit{arrangement} $\A(\H)$ of $\H$. For more details, we refer to \cite{edelsbrunner1986constructing,edelsbrunner1993zone} or standard monographs in combinatorial geometry (e.g., \cite{edelsbrunner1987algorithms}). \Cref{lem:HA-counting} below gives a standard counting bound for hyperplane arrangements.

\begin{lemma}[\cite{edelsbrunner1986constructing}]\label{lem:HA-counting}
	Let $\H$ be a finite set of $p$ hyperplanes in $\Re^q$. Then the arrangement $\A(\H)$ can be constructed in $\O(p^q)$ time and contains $\O(p^\tau)$ $\tau$-cells for each $\tau\in[q]$.
	Moreover, if all hyperplanes in $\H$ pass through the origin, then these two bounds tighten to $\O((p-1)^{q-1})$  and $\O((p-1)^{\tau-1})$, respectively.
\end{lemma}

We now apply \Cref{lem:HA-counting} to the hyperplane arrangement induced by 
\begin{align}\label{eq:HA-pairwise}
	H_{ij} = \left\{c\in \Re^r:   \left(A^{\top}c\right)_i - \left(A^{\top}c\right)_j=0 \right\},\qquad \forall 1 \le i < j \le n.    
\end{align}
\begin{lemma}\label{lem:HA}
	The hyperplanes in \eqref{eq:HA-pairwise} partition $\R^r$ into $\O(n^{2r-2})$ cells. Moreover, for each cell $R$, there exists a  permutation $\pi\in\Pi_n$ such that $A^\top c\in\set{Z}(\pi)$ for all $c\in R$.
\end{lemma}
\begin{pf}
	Let $\H=\{H_{ij}\}_{1\le i < j \le n}$ denote the hyperplanes given in~\eqref{eq:HA-pairwise}.  Since $|\H|=n(n-1)/2$ and all hyperplanes in $\H$ pass through the origin, \Cref{lem:HA-counting} implies that the number of cells is upper bounded by $\O(|\H|^{r-1})=\O(n^{2r-2})$.
	
	Each cell $R$ is characterized by $\bigcap_{1\le i<j\le n}\tilde H_{ij}$, where $\tilde H_{ij}\in\{H_{ij},\,H_{ij}^{>},\,H_{ij}^{<}\}$. By  construction, for any fixed  cell $R$ and any pair $i<j$, the sign of $(A^\top c)_i-(A^\top c)_j$ is constant over $c\in R$. This implies that the relative ordering between $(A^\top c)_i$ and $(A^\top c)_j$ is invariant within $R$. Taken over all pairs, these comparisons determine an overall ordering of $A^\top c$ throughout $R$. Hence, there exists a permutation $\pi\in\Pi_n$ such that $A^\top c\in \set{Z}(\pi)$ for all $c\in R$, completing the proof. 
\end{pf}

\subsection{Support construction}
Based on the cell decomposition in \Cref{lem:HA}, we now turn to support construction and deriving complexity results for \eqref{eq:general-concave}. 
We begin with showing that a fixed ordering can yield $\O(n^2)$ distinct supports.
\begin{lemma}\label{lem:number-per-order}
For any $\sigma\in\Pi_n$, the set $\{\supp(x): x\in \Z(\sigma)\}$ contains $\O(n^2)$ supports.
\end{lemma}
\begin{pf}
Since every vector $x\in \Z(\sigma)$  is sorted by $\sigma$, 
its zero entries (if any) appear consecutively and form a contiguous block under this permutation. The support of $x$ is uniquely determined by  this zero block. Specifically, if $x$ contains no zero entries, there is exactly one support $\supp(x)=[n]$. Otherwise, there exist indices $1\le t_1\le t_2\le n$ such that 
\begin{align*}
	x_{\sigma(1)}\ge \cdots \ge x_{\sigma(t_1-1)} >0, \ \   x_{\sigma(t_1)}=\cdots =x_{\sigma(t_2)}=0,  \ \ 0>x_{\sigma(t_2+1)}\ge \cdots \ge x_{\sigma(n)}.  
\end{align*}
For any fixed pair $(t_1, t_2)$, the support of $x$ is then uniquely determined. Since there are $\O(n^2)$ possible choices of $(t_1, t_2)$, the conclusion follows.
\end{pf}

In the rest of this paper, we make the following assumption.
\begin{assumption}\label{assumption:permutation-oracle-time}
The permutation mapping  $\permMap$ in \Cref{def:map} can be accessed in polynomial time; we denote this cost by $\textsf{T}_2$. 
\end{assumption}
If the feasible region $\X$ is standard comonotone,  then \Cref{assumption:permutation-oracle-time} is trivial because of $\textsf{T}_2=
\O(1)$.
We recall that $\textsf{T}_1$ denotes the time required to solve a fixed-support subproblem in Assumption~\ref{assume:supp}.
We are now ready to present the main result in this section.
\begin{theorem}\label{them:supp} Suppose that $\set{X}$ is comonotone. Then the following hold:
\begin{enumerate}[(i)]
	\item There exists a collection of $\mathcal{O}(n^{2r})$ candidate supports for \eqref{eq:general-concave}, among which at least one support is optimal.
	\item  Under \Cref{assume:supp} and \Cref{assumption:permutation-oracle-time},  problem
	\eqref{eq:general-concave} admits a polynomial-time algorithm with complexity  $\O(n^{2r-2}\cdot\textsf{T}_2 +  n^{2r}\cdot\textsf{T}_1)$.
\end{enumerate}
\end{theorem}
\begin{pf}
Our proof is constructive and includes two parts.

\noindent \textbf{Part I.} 
We first prove (i). By \Cref{lem:HA}, the hyperplanes in \eqref{eq:HA-pairwise} partition $\Re^r$ into $K=\O(n^{2r-2})$ cells, which we denote by $R_1, \cdots, R_K$. In addition, for each $k\in [K]$, there exists a permutation $\pi^k\in \Pi_n$ such that $A^\top c\in\Z(\pi^k)$ for all $c\in R_k$. Let $\permMap$ be the permutation mapping associated with the comonotone set $\set{X}$ and $\sigma^k=\permMap(\pi^k)$ for all $k\in[K]$.

According to \Cref{prop:linear}, there exists a cost vector $c^*\in \Re^r$  such that every optimal solution to \eqref{eq:linear} with $c=c^*$ is also optimal for \eqref{eq:general-concave}. Since $\{R_k\}_{k\in[K]}$ forms a partition of $\R^r$, the vector $c^*$ lies in a certain cell $R_{\bar{k}}$ with $\bar{k}\in[K]$, implying $A^\top c^*\in\Z(\pi^{\bar{k}})$ by \Cref{lem:HA}. 
Combining the above with the optimality condition~\eqref{eq:optimality-general-comonotonicity}, we deduce that $\Z(\permMap(\pi^{{\bar{k}}}))=\Z(\sigma^{{\bar{k}}})$ contains an optimal solution to \eqref{eq:general-concave}, and therefore so does $\bigcup\limits_{k\in [K]}\Z(\sigma^{k})$. This allows us to equivalently convert \eqref{eq:general-concave}  into
\begin{align}\label{eq:linear_concave1}
	\max_{x\in \X} \bigg\{f(A^{\top} x): x\in \bigcup\limits_{k\in [K]} \Z(\sigma^{k})\bigg\}.
\end{align}
By \Cref{lem:number-per-order}, each fixed ordering $\Z(\sigma^k)$ can yield at most $\O(n^2)$ possible supports.
Since \eqref{eq:linear_concave1} involves $K=\O(n^{2r-2})$ orderings, the total number of candidate supports is $\O(n^{2r})$. This proves (i).

\noindent \textbf{Part II.} We next prove (ii). Let $S_1,\ldots,S_L$ be the support candidates constructed above, where
$L=\O(n^{2r})$. Then \eqref{eq:linear_concave1} admits the equivalent reformulation
\begin{align}\label{eq:linear_concave2}
	\max_{\ell \in [L]} \max_{x\in \X} \bigg\{f(A^{\top} x):\supp(x)=S_{\ell}\bigg\}.
\end{align}

The running time consists of two parts: constructing $\{S_\ell\}_{\ell\in[L]}$ and solving all fixed-support subproblems, i.e., the inner maximizations in \eqref{eq:linear_concave2}. Because accessing $\sigma^k=\permMap(\pi^k)$ costs $\textsf{T}_2$ per cell under \Cref{assumption:permutation-oracle-time} and there are $\O(n^{2r-2})$ cells, the total  time for support construction is $\O(n^{2r-2}\cdot \textsf{T}_2+ n^{2r})$. Moreover, under \Cref{assume:supp}, each fixed-support subproblem in \eqref{eq:linear_concave2} can be solved in $\textsf{T}_1$ time. Thus, addressing all subproblems requires $ L\cdot\textsf{T}_1 = \O(n^{2r}) \cdot \textsf{T}_1$ time. Consequently, the overall complexity is given by $\O\left(n^{2r-2} \cdot \textsf{T}_2 + n^{2r} \cdot \textsf{T}_1\right)$. This completes the proof.
\end{pf}

When $\set{X}$ represents a matroid or the bases of a matroid, the complexity bounds from \Cref{them:supp} can be further improved as shown in \Cref{cor:matroid} below, which recovers the result of \cite[Theorem 1.4]{onn2003convex} from a different perspective.

\begin{corollary}\label{cor:matroid}
If $\X \subseteq \{0,1\}^n$ encodes a matroid or the bases of a matroid,  then problem
\eqref{eq:general-concave} admits a polynomial-time algorithm with complexity  $\O(n^{2r-1} \log n)$.
\end{corollary}
\begin{pf}
Suppose that $\set{X}$ represents the bases of a matroid. Then, by \Cref{lem:greedy}, as long as the cost vector $A^{\top} c$ of \eqref{eq:linear} is sorted by the same permutation, the best-in greedy algorithm returns the same optimal solution $x^*$ to \eqref{eq:linear}.  It is thus enough to evaluate one support per permutation cell. Since the arrangement in \Cref{them:supp} has $\O(n^{2r-2})$ cells, this yields $\O(n^{2r-2})$ candidate supports. Furthermore, \Cref{assume:supp} is satisfied due to $\set{X}\subseteq\{0,1\}^n$. Hence,  the claimed complexity follows from $\textsf{T}_2=\O(n\log n)$ by \Cref{lem:greedy}.

We next turn to the case where $\set{X}$ represents a matroid. 
In this setting, the output of the best-in greedy algorithm also depends on the sign of the cost vector, as shown in the proof of \Cref{lem:greedy}. Motivated  by this,
we refine the hyperplane arrangement in \Cref{them:supp} by including the coordinate hyperplanes
\[
\hat H_i=\{c\in\R^r:\ (A^\top c)_i=0\},\qquad \forall i\in[n].
\]
On each induced cell $R$, both the relative ordering of the entries of $A^\top c$ (from~\eqref{eq:HA-pairwise}) and the sign pattern of $A^\top c$ (from $\{\hat H_i\}_{i\in [n]}$) remain fixed. Thereby, the solution $\bar x$ constructed in \Cref{prop:matroid-comonotone} is optimal for all $c\in R$, implying that each cell contributes a single support candidate. Since the total number of hyperplanes in $\Re^r$ is $n(n-1)/2+n=\O(n^2)$, according to \Cref{lem:HA-counting}, the refined hyperplane arrangement still yields $\O(n^{2r-2})$ cells and results in an overall complexity of $\O(n^{2r-2})\cdot \textsf{T}_2=\O(n^{2r-1}\log n)$.
\end{pf} 

\section{Improved complexity via lifting under standard comonotonicity}\label{sec:inv}
Lifting is a widely-used technique in mixed-integer programming to simplify algebraic representation by working in an extended space
where the object in question is easier to describe 
and then interpreting the original object as its projection \cite{yannakakis1988expressing, martin1990polyhedral, conforti2010extended, goemans2015smallest, Han2023}. In this section, we present a lifting technique of this kind to refine the complexity analysis of \eqref{eq:general-concave} when its feasible region $\X$ is standard comonotone. 
Rather than constructing the hyperplane arrangement directly in the original parameter space $\Re^r$ of $c$, we augment it with additional parameters and build an arrangement in a lifted space. 
Projecting the cells of the lifted arrangement back onto the original space yields the same partition as the pairwise arrangement~\eqref{eq:HA-pairwise}. Crucially, the lifted arrangement is defined by only $\O(n)$ hyperplanes, as opposed to $\O(n^2)$ in \eqref{eq:HA-pairwise}. 
This reduction in turn leads to improved complexity bounds under standard comonotonicity.

To illustrate the lifting idea, consider a specific standard comonotone set $\X=\{x\in [0,1]^n: e^{\top} x= s\}$, where $s\in[n]$. Since \eqref{eq:linear} is a linear program,  strong duality guarantees the existence of an optimal multiplier $\lambda$ such that \eqref{eq:linear} is  equivalent to
\[\max_{x\in [0,1]^n} c^{\top}Ax-\lambda e^{\top}x+\lambda s = \max_{x\in [0,1]^n} \left(A^{\top}c- \lambda e\right)^{\top}x+\lambda s. \]
Clearly,  any optimal solution $x^*$ to \eqref{eq:linear} must satisfy the thresholding rule: for any $i\in [n]$,
\begin{enumerate}[(i)]
\item If $\left(A^{\top}c\right)_i>\lambda$, then $x_i^*=1$.
\item If $\left(A^{\top}c\right)_i<\lambda$, then $x_i^*=0$.
\end{enumerate}

The thresholding rule motivates us to consider the lifted arrangement in $\R^r\times\R$ induced by the $n$ hyperplanes $\{(c,\lambda)\in\R^r\times\R:\,(A^\top c)_i-\lambda=0  \}$, $i\in [n]$. 

Figure~\ref{fig:lifted-vs-original-HA} illustrates an instance of an hyperplane arrangement and its lifted variant with $r=3$ and $n=4$. 	For visualization, we only depict the cross-section obtained by fixing $c_3=1$. 
In the original space of cost vectors $c$, the pairwise arrangement in Fig.~\ref{fig:original-HA} is induced by $\binom{n}{2}=6$ lines.  
In contrast, in the lifted space shown in Fig.~\ref{fig:lifted-HA}, it suffices to consider $n=4$ planes  which correspond to the four facets of a tetrahedron. 
The shaded region in Fig.~\ref{fig:original-HA} depicts the projection of this tetrahedron onto the original $c$-space. 
We formalize the above lifting idea in the remainder of the section.

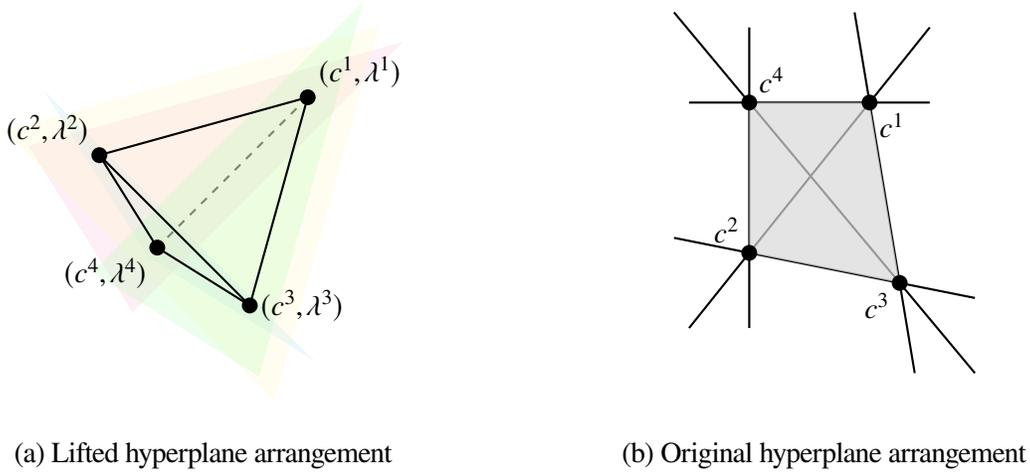
\begin{figure}[h]
\centering	\begin{subfigure}[t]{0.5\textwidth}
	\centering
	\def\ext{1.8} 
	\begin{tikzpicture}[scale=2,>=Stealth,baseline=-1.5cm]
		\tikzset{every path/.style={-}}
		\coordinate (A) at (1, 1, -1);
		\coordinate (B) at (0, 1, 0);
		\coordinate (C) at (1, 0, 0);
		\coordinate (D) at (0, 0, -1);
		
		\coordinate (G_bcd) at (barycentric cs:B=1,C=1,D=1); 
		\fill[cyan!30, fill opacity=0.25] 
		($(G_bcd)!\ext!(B)$) -- ($(G_bcd)!\ext!(C)$) -- ($(G_bcd)!\ext!(D)$) -- cycle;
		
		\coordinate (G_abd) at (barycentric cs:A=1,B=1,D=1);
		\fill[magenta!30, fill opacity=0.25] 
		($(G_abd)!\ext!(A)$) -- ($(G_abd)!\ext!(B)$) -- ($(G_abd)!\ext!(D)$) -- cycle;

		\draw[dashed, thick, black!80] (A) -- (D);

		\coordinate (G_abc) at (barycentric cs:A=1,B=1,C=1); 
		\fill[yellow!30, fill opacity=0.25] 
		($(G_abc)!\ext!(A)$) -- ($(G_abc)!\ext!(B)$) -- ($(G_abc)!\ext!(C)$) -- cycle;
		
		\coordinate (G_acd) at (barycentric cs:A=1,C=1,D=1); 
		\fill[green!30, fill opacity=0.25] 
		($(G_acd)!\ext!(A)$) -- ($(G_acd)!\ext!(C)$) -- ($(G_acd)!\ext!(D)$) -- cycle;
		
		\draw[thick, black] (B) -- (D); 
		\draw[ thick, black] (C) -- (D);
		\draw[ thick, black] (A) -- (B);
		\draw[ thick, black] (B) -- (C);
		\draw[ thick, black] (C) -- (A);
		
		\node[anchor=south west, xshift=0pt] at (A) {$(c^1,\lambda^1)$};
		\node[anchor=south east, xshift=0pt] at (B) {$(c^2,\lambda^2)$};
		\node[anchor=west, xshift=0pt] at (C) {$(c^3,\lambda^3)$};
		\node[anchor=north east, yshift=0pt] at (D) {$(c^4,\lambda^4)$};
		
		\foreach \p in {A,B,C,D}
		\fill[black] (\p) circle (1.5pt); 
		
	\end{tikzpicture}
	
	\caption{Lifted hyperplane arrangement}\label{fig:lifted-HA}
\end{subfigure}\hfill
\begin{subfigure}[t]{0.5\textwidth}			
	\centering
	\def\ext{1.5} 
	\begin{tikzpicture}[scale=2,>=Stealth,baseline=-2.2cm]
		\tikzset{every path/.style={-}}
		\coordinate (A) at (0.8, 1 );
		\coordinate (B) at (0, 1);
		\coordinate (C) at (1, -0.2);
		\coordinate (D) at (0, 0);
		
		\draw[thick]($(A)!\ext!(D)$) -- ($(D)!\ext!(A)$); 
		\draw[thick]($(C)!\ext!(B)$) -- ($(B)!\ext!(C)$); 
		\draw[thick]($(A)!\ext!(B)$) -- ($(B)!\ext!(A)$); 
		\draw[thick]($(A)!\ext!(C)$) -- ($(C)!\ext!(A)$); 
		\draw[thick]($(C)!\ext!(D)$) -- ($(D)!\ext!(C)$);
		\draw[thick]($(D)!\ext!(B)$) -- ($(B)!\ext!(D)$);  
		\fill[black!15, opacity=0.7] (A) -- (B) -- (D) -- (C) -- cycle;
		\node[anchor=north west] at (A) {$c^1$};
		\node[anchor=south west] at (B) {$c^4$};
		\node[anchor=north east] at (C) {$c^3$};
		\node[anchor=south east] at (D) {$c^2$};
		\foreach \p in {A,B,C,D}
		\fill[black] (\p) circle (1.5pt);
	\end{tikzpicture}
	\caption{Original hyperplane arrangement}\label{fig:original-HA}
\end{subfigure}
\caption{Comparison of hyperplane arrangements ($r=3,n=4$, slice $c_3=1$)}\label{fig:lifted-vs-original-HA}
\end{figure}

\subsection{New optimality conditions}
We first derive refined optimality conditions for \eqref{eq:linear} over a general standard \comonotone feasible set $\X$. The conditions show that comparing each component of $A^\top c$ against two thresholding parameters can determine the sign pattern of an optimal solution.

\begin{definition}\label{def:set}
For any vector $c\in \Re^r$ and  parameters $\underline \lambda \le \overline \lambda$, define the set
\[
\Q\left(c, \overline \lambda, \underline{\lambda}\right) \defeq 
\left\{ 
x \in \Re^n : \text{(a) sign constraints, (b) ordering constraints}
\right\},
\]
where
\begin{enumerate}[(a)]
	\item Sign constraints: For all $i\in [n]$,
	\[
	\begin{array}{ll}
		x_{i} > 0  \text{ if } \left( A^{\top} c \right)_i > \overline \lambda, \quad
		x_{i} = 0  \text{ if } \underline{\lambda} < \left( A^{\top} c \right)_i < \overline \lambda, \quad
		x_{i} < 0  \text{ if } \left( A^{\top} c \right)_i < \underline{\lambda},\\
		\vspace{-0.5em}\\
		x_i\ge0 \text{ if } \left( A^{\top} c \right)_i = \overline{\lambda} \text{ and } \overline\lambda>\underline\lambda, \quad  x_i \le0 \text{ if } \left( A^{\top} c \right)_i = \underline{\lambda} \text{ and } \overline\lambda>\underline\lambda.
	\end{array}
	\]
	\item Ordering constraints: For all $i,j\in [n]$ with $i<j$,
	\[x_i \ge  x_j  \text{ if } \left( A^{\top} c \right)_i = \left( A^{\top} c \right)_j\in\left\{\overline{\lambda},\underline\lambda\right\}.  \]
\end{enumerate}
For convenience, we treat $\Q(c, \overline \lambda, \underline \lambda)=\emptyset$ if $\underline \lambda>\overline \lambda$ .
\end{definition}

\begin{lemma}\label{lem:sol_perm_inv}
Suppose that $\X$ is standard \comonotone. Then, for any $c\in \Re^r$, there exist parameters $\underline\lambda \le \overline \lambda$ such that $\Q\left(c, \overline \lambda, \underline{\lambda}\right)$ contains an optimal solution $x^*$ of  \eqref{eq:linear}, i.e.,
\[ x^* \in\Q\left(c, \overline \lambda, \underline{\lambda}\right),\]
{whenever \eqref{eq:linear} attains the optimum.}
\end{lemma}

\begin{pf}
The vector $A^\top c$ may admit multiple sorting permutations $\pi\in\Pi_n$ due to ties. 
We consider one such $\pi$ using the lexicographical tie-breaking rule: if $(Ac)_i=(Ac)_j$ and $i<j$, then $\pi_i<\pi_j$.   By standard comonotonicity of $\set{X}$, \eqref{eq:linear} admits an optimal solution $x^*$ such that for all $1\le i<j\le n$, 
\begin{equation}\label{eq:tie-breaking}
	x_i^* \ge x_j^*   \text{ if }\left(A^{\top}c\right)_i = \left(A^{\top}c\right)_j.
\end{equation}

Next,  we partition the indices according to the sign of $x^*$:
\[\I^{+} = \{i: x^*_i>0\}, \ \ \I^{0}=\{i: x^*_i=0\}, \ \  \I^{-}=\{i: x^*_i<0\}, \] and introduce
\[ \underline{\lambda}=\max_{i\in\I^-} \left(A^{\top}c\right)_i, \ \  \overline\lambda =\min_{i\in \I^+} \left(A^{\top}c\right)_i,\]
with the conventions $\underline{\lambda}=\min_{i\in[n]}(A^\top c)_i$ if $\I^-=\emptyset$ and  $\overline{\lambda}=\max_{i\in[n]}(A^\top c)_i$ if $\I^+=\emptyset$.  It is evident that $x_i^*>x_j^*>x_{\ell}^*$ for all $i\in\I^+, j\in\I^0,\ell\in\I^-$. By  \Cref{thm:comonotone-symmetry}~(iii), which holds as a necessary condition for standard comonotonicity even when $\set X$ is not compact,  one has that 
\begin{equation}\label{eq:ordering-two-lambda}
	(A^\top c)_i\ge \overline{\lambda}\ge(A^\top c)_j \ge \underline{\lambda}\ge (A^\top c)_{\ell}\quad \forall i\in\I^+, j\in\I^0,\ell\in\I^-.
\end{equation}

Finally, to show $x^*\in\Q(c,\overline{\lambda},\underline{\lambda})$, we check the conditions in \Cref{def:set}:
\begin{list}{}
	{%
		\setlength{\leftmargin}{0pt}%
		\setlength{\labelwidth}{0pt}%
		\setlength{\labelsep}{0pt}%
		\setlength{\itemindent}{0pt}%
	}
		\item $\bullet$ If $(A^\top c)_i> \overline{\lambda}$, then $i\notin\I^-\cup\I^0$ in view of \eqref{eq:ordering-two-lambda}, implying $x_i^*>0$. Similarly, one can verify that $(A^\top c)_i<\underline{\lambda}$ implies $x_i^*<0$, and $\underline{\lambda}<(A^\top c)_i<\overline{\lambda}$ implies $x_i^*=0$.
	\item $\bullet$ If $\overline{\lambda}>\underline{\lambda}$ and $ \left( A^{\top} c \right)_i =\overline\lambda$, then one also has $ \left( A^{\top} c \right)_i>\underline{\lambda}$. Thus, \eqref{eq:ordering-two-lambda} implies $i\not\in\I^-$, i.e., $x_i^*\ge0$. Similarly, $\overline{\lambda}>\underline{\lambda}$ and $ \left( A^{\top} c \right)_i =\underline\lambda$ imply $x_i^*
	\le 0$.
	\item $\bullet$ The ordering constraints in \Cref{def:set} hold by \eqref{eq:tie-breaking}.
\end{list}
This completes the proof.
\end{pf}

The optimality condition in \Cref{lem:sol_perm_inv} is more straightforward than the one in \Cref{lem:HA}. The former explicitly determines the sign pattern of  an optimal solution to \eqref{eq:linear}, whereas the latter only provides the ordering information.

\subsection{Extended space partition and improved complexity}
In this subsection, building on the optimality condition in \Cref{lem:sol_perm_inv}, we propose a new partition of the extended parameter space $\R^{r}\times \R\times \R$ that corresponds to $c$, $\underline \lambda$, and $\overline \lambda$, respectively. We then show how this partition leads to an improved complexity bound under standard comonotonicity.

The partition of the extended space is induced by the following $2n+1$ hyperplanes
\begin{equation}\label{eq:HA-threshold}
\begin{aligned}
	H_i &= \left\{\left(c, \underline \lambda, \overline \lambda\right)\in \Re^{r}\times \Re\times \Re: \left( A^{\top} c \right)_i= \overline \lambda \right\}, \ \ \forall i\in [n],\\
	H_{i+n} &= \left\{\left(c, \underline \lambda, \overline \lambda\right) \in \Re^{r}\times \Re\times \Re: \left( A^{\top} c \right)_i= \underline \lambda \right\}, \ \ \forall i\in [n],\\
	H_{2n+1}&=\left\{\left(c, \underline \lambda, \overline \lambda\right)\in \Re^{r}\times \Re\times \Re: \overline \lambda= \underline \lambda \right\}.
\end{aligned}
\end{equation}
We need \Cref{lem:number-per-order-nn} below to bound the number of supports per cell.
\begin{lemma}\label{lem:number-per-order-nn}
For any $\sigma\in \Pi_n$, the set $\left\{\supp(x):\,x\in \Z(\pi)\cap\R^n_+\right\}$ contains $\O(n)$ supports. 
\end{lemma}
\begin{pf} 
If $x \in \Z(\pi)\cap\R^n_+$ is strictly positive, its support is exactly $[n]$.
If $x\in \Z(\pi)\cap\R^n_+$ contains zeros, 
by nonnegativity, the vector must satisfy
\[x_{\sigma(1)}\ge \cdots \ge x_{\sigma(t)} >0, \ \   x_{\sigma(t+1)}= \cdots = x_{\sigma(n)} =0 \]
for some index $t \in [n]$, which yields $n$ distinct supports. 
\end{pf}

\begin{theorem}\label{thm:standard-comonotone-complexity}
Suppose that $\X$ is standard \comonotone. Then the following statements hold:
\begin{enumerate}[(i)]
	\item There exists a collection of  {$\mathcal{O}(n^{r+1})$} candidate supports for \eqref{eq:general-concave}, among which at least one support is optimal.
	\item Under \Cref{assume:supp},  problem \eqref{eq:general-concave} admits a polynomial-time algorithm with complexity {$\O(n^{r+1} \cdot \textsf{T}_1)$}.
\end{enumerate}
\end{theorem}
\begin{pf} Denote by $\H=\left\{H_t\right\}_{t\in [2n+1]}$ the hyperplanes in \eqref{eq:HA-threshold}. Let $\A(\H)=\{R_k\}_{k\in[K]}$ be the resulting arrangement.  Within each cell $R_k$,  the signs of
\[ \left( A^{\top} c \right)_i-  \overline \lambda, \forall i\in [n], \ \ \left( A^{\top} c \right)_i-  \underline \lambda, \forall i\in [n], \text{ and } \overline \lambda -\underline \lambda  \]
are constant. 
By \Cref{def:set}, the set $\Q(c, \overline \lambda, \underline \lambda)$ depends only on these sign patterns and thus is invariant within each cell.  More specifically, the following index sets $\I^+_k$, $\I^0_k$, $\I^-_k$, $\I^{=\overline \lambda}_k$, and $\I^{=\underline \lambda}_k$ remain fixed for any $(c, \overline \lambda, \underline \lambda)\in R_k$, where
\begin{align*}
	\begin{array}{ll}
		\I^+_k  =\left\{ i: \left( A^{\top} c \right)_i > \overline \lambda \right\} , \quad
		\I^0_k  = \left \{i: \underline{\lambda} < \left( A^{\top} c \right)_i < \overline \lambda \right\}, \quad
		\I^-_k  =\left\{i: \left( A^{\top} c \right)_i < \underline{\lambda}\right\}, \\
		\I^{=\overline{\lambda}}_k  =\left\{ i: \left( A^{\top} c \right)_i = \overline \lambda \right\} , \quad
		\I^{=\underline{\lambda}}_k  = \left\{i: \left( A^{\top} c \right)_i = \underline \lambda\right\}. 
	\end{array}
\end{align*}
Based on these index sets,
we define the cell-dependent set in $\R^n$:
\begin{equation}\label{eq:set_Z}
	\tilde\Upsilon^k\defeq \left\{x\in \Re^n:
	\begin{aligned}
		&x_i>0, \forall i\in \I_k^+, \ \  x_i=0, \forall i\in \I_k^0, \ \ x_i<0, \forall i\in \I_k^-\\ 
		&x_i\ge x_j, \forall i,j\in \I_k^{=\overline \lambda}  \text{ with } i<j \text{ or } \forall i,j\in \I_k^{=\underline \lambda}  \text{ with } i<j
	\end{aligned}\right\}.
\end{equation}
Since the sign of $\overline\lambda-\underline\lambda$ is constant on $R_k$, define $\Upsilon^k$ by setting $\Upsilon^k=\emptyset$ if $\overline\lambda<\underline\lambda$ on $R_k$, $\Upsilon^k=\tilde\Upsilon^k$ if $\overline\lambda=\underline\lambda$ on $R_k$, and 
\[\Upsilon^k=\tilde{\Upsilon}^k\cap\bigg\{ x\in\R^n:\,x_i\ge0\,\forall i\in \I_k^{=\overline \lambda}  \text{ and }  x_i\le0\,\forall i\in \I_k^{=\underline \lambda} \bigg\}\]
if $\overline{\lambda}>\underline{\lambda}$ on $R_k$. By construction, we have $\Upsilon^k=\Q(c, \overline \lambda, \underline \lambda)$ for every $(c, \overline \lambda, \underline \lambda)\in R_k$.

Combining \Cref{prop:linear} with \Cref{lem:sol_perm_inv}, there exist $c^*$  and $\underline \lambda^* \le \overline \lambda^*$ such that 
$\Q(c^*, \overline \lambda^*, \underline \lambda^*)$
contains an optimal solution of  \eqref{eq:general-concave}. Since $(c^*, \overline \lambda^*, \underline \lambda^*)$ must belong to a certain cell,  \eqref{eq:general-concave} can be equivalently converted into
\begin{align}\label{eq:linear_concave3}
	\max_{x\in \X} \left\{f(A x): x\in \bigcup\limits_{k\in [K]}\Upsilon^k\right\}.   
\end{align}

Next, we bound the number of supports contributed by each nonempty $\Upsilon^k$.  By  \eqref{eq:set_Z}, given $x\in \Upsilon^k$, the signs of $x_i$ are fixed for any $i\in \I^+_k\cup \I^0_k \cup \I^-_k$. While the entries indexed by $ \I_k^{=\overline\lambda}\cup \I_k^{=\underline \lambda}$ do not have fixed signs, their relative order is fixed by construction. This leads to three cases:
\begin{itemize}
	\item \emph{Case 1: $\I_k^{=\overline \lambda}\cup \I_k^{=\underline \lambda}=\emptyset$}. The number of such cells is at most $\O((2n)^{r+1})=\O(n^{r+1})$ by \Cref{lem:HA-counting}. In this case,  the corresponding set $\Upsilon^k$ can only contribute one single support candidate.
	\item \emph{Case 2: exactly one of $\I_k^{=\overline \lambda}$ and $ \I_k^{=\underline \lambda}$ is nonempty.} In this case, one must have $\overline{\lambda}>\underline{\lambda}$ and $\dim(R_k)\le r+1$.  The number of such $(r+1)$-cells is $\O((2n)^r)=\O(n^r)$ by \Cref{lem:HA-counting}. In addition, applying \Cref{lem:number-per-order-nn} to the nonempty tie set,  each such $\Upsilon^k$ can yield $\O(n)$ supports. 
	\item \emph{Case 3: both $\I_k^{=\overline \lambda}$ and $ \I_k^{=\underline \lambda}$ are nonempty.} It follows that at least two non-parallel hyperplanes in \eqref{eq:HA-threshold} are active on the cell $R_k$.
	In this case, one has $\dim(R_k)\le r$. The number of such $r-$cells is $\O((2n)^{r-1})=\O(n^{r-1})$ by \Cref{lem:HA-counting}. 
	Applying \Cref{lem:number-per-order} to $\I_k^{=\overline \lambda}\cup \I_k^{=\underline \lambda}$, each $\Upsilon(k)$ in \eqref{eq:linear_concave3} can yield $\O(n^2)$ supports. 
\end{itemize}
Combining the three cases, the total number of distinct supports for \eqref{eq:linear_concave3} is $\O(n^{r+1})+\O(n^r)\cdot\O(n)+\O(n^{r-1})\cdot\O(n^2)=\O(n^{r+1})$.  The proof for (ii) is identical to \textbf{Part II} of \Cref{them:supp} and thus omitted. 
\end{pf}

Compared with \Cref{them:supp}, \Cref{thm:standard-comonotone-complexity} improves the support complexity by a square-root factor, reducing it from $\O(n^{2r})$ to $\O(n^{r+1})$.

\subsection{Improved complexity under additional sign conditions}
\label{subsec:improve_perm_inv}
In this subsection, we  show that additional sign structures can lead to further complexity reductions for \eqref{eq:general-concave} over standard \comonotone sets. 
\subsubsection{Nonnegative and standard \comonotone sets}
When $\X$ is nonnegative and standard \comonotone, any optimal solution of \eqref{eq:linear} contains no negative entries. Accordingly, the two thresholds in \Cref{def:set} introduced to separate positive and negative entries are no longer both necessary. Indeed, it suffices to keep a single threshold associated with positive entries, leading to the simplified variant of $\Q$ in \Cref{def:set} below.

\begin{definition}\label{def:set_nonneg}
For any vector $c\in \Re^r$ and a parameter $\lambda$, define
\[
\Q_+ \left(c,  \lambda \right) \defeq 
\left\{ 
x \in \Re_+^n : \text{(a) sign constraints, (b) ordering constraints}
\right\},
\]
where
\begin{enumerate}[(a)]
	\item Sign constraints: For any $i\in [n]$,
	\[
	\begin{array}{ll}
		x_{i} > 0  \text{ if } \left( A^{\top} c \right)_i >  \lambda, \quad
		x_{i} = 0  \text{ if }  \left( A^{\top} c \right)_i < \lambda.
	\end{array}
	\]
	\item Ordering constraints: For any $i,j\in [n]$ with $i<j$,
	\[\begin{array}{ll}
		x_i \ge  x_j \ge 0& \text{if } \left( A^{\top} c \right)_i = \left( A^{\top} c \right)_j=\lambda.
	\end{array} \]
\end{enumerate}
\end{definition}

The set $\Q_+$ characterizes the optimality conditions of~\eqref{eq:linear} when $\set{X}$ is nonnegative and standard comonotone. In turn, the conditions on $(c,\lambda)$ in \Cref{def:set_nonneg} motivate partitioning the parameter space $\R^r\times \R$ using the $n$ hyperplanes below
\begin{equation}\label{eq:HA-threshold-nn}
H_i = \left\{\left(c, \lambda \right)\in \Re^{r}\times \Re: \left( A^{\top} c \right)_i=  \lambda \right\}, \ \ \forall i\in [n].
\end{equation}
Consequently,  the simpler partition yields fewer cells  than  \eqref{eq:HA-threshold} and leads to an improved complexity bound in \Cref{prop:nonneg-standard-comonotone}. 

\begin{proposition}
\label{prop:nonneg-standard-comonotone}
Suppose that $\X$ is nonnegative and standard \comonotone. Then 
\begin{enumerate}[(i)]
	\item There exists a collection of  $\mathcal{O}(n^{r})$ candidate supports for \eqref{eq:general-concave}, among which at least one support is optimal.
	\item Under \Cref{assume:supp},  problem \eqref{eq:general-concave} admits a polynomial-time algorithm with complexity $\O(n^{r} \cdot \textsf{T}_1)$.
\end{enumerate}
\end{proposition}
\begin{pf}
Since the argument follows that of \Cref{thm:standard-comonotone-complexity}, we highlight only the necessary changes. Let $\A=\{R_k\}_{k\in[K]}$ be the arrangement induced by the hyperplanes in \eqref{eq:HA-threshold-nn}. By adapting the analysis of \Cref{lem:sol_perm_inv} and 
dropping the parameters $\underline \lambda$ therein,  one can deduce that for any $c\in \Re^r$, there exists a parameter $\lambda$ such that one optimal solution of \eqref{eq:linear} lies in $\Q_+(c, \lambda)$. Moreover, by construction,  $\Q_+(c, \lambda)$ is a constant set (denoted by $\Upsilon^k$) within each cell $R_k$. Consequently, as shown in \Cref{thm:standard-comonotone-complexity},  $\cup_{k\in[K]}\Upsilon^k$ contains an optimal solution to \eqref{eq:general-concave}. It remains to bound the number of possible supports for  $\cup_{k\in[K]}\Upsilon^k$. We first notice that \emph{Case~3} in the proof of \Cref{thm:standard-comonotone-complexity} does not exist. Moreover, since the lifted parameter space has one fewer dimension, the support count estimated in the analysis of \Cref{thm:standard-comonotone-complexity} reduces to $\O(n^r)$ by \Cref{lem:HA-counting}. We thus complete the proof.
\end{pf}

\subsubsection{Sign-invariant and standard \comonotone sets} A set $\X\subseteq \Re^n$ is called \textit{sign-invariant} if each $\bar x\in \X$ implies $x\in \X$ for all $x$ satisfying $|x|=|\bar x|$. Suppose that $\X$ is sign-invariant and standard \comonotone. Then, for any $c\in \Re^r$ and $x\in \X$, we can construct $\bar x\in \X$ such that
\[c^{\top} Ax \le  c^{\top} A\bar x= \left(\left|A^{\top}c\right|\right)^T |x|. \]
Specifically, for each $i\in [n]$, we let $\bar x_i= x_i$ if $(A^{\top}c)_ix_i\ge 0$, and $\bar x_i= -x_i$ otherwise. 
Consequently, if $x^*\in\set{X}$ is optimal to \eqref{eq:linear}, then $|x^*|$ must be optimal to
\begin{equation}\label{eq:sign_linear}
\max_{x\in \X} \left(\left|A^{\top}c\right|\right)^T x.
\end{equation}

Because $|x^*|\in\set{X}$ is nonnegative and shares the same support as $x^*$, most of the analysis developed earlier for nonnegative standard comonotone sets directly applies, with \eqref{eq:sign_linear} in place of \eqref{eq:linear}. The main difference is that the objective in \eqref{eq:sign_linear} involves absolute values of the cost vector. Accordingly, we modify $\Q_+(c, \lambda)$ by replacing $A^{\top}c$ in \Cref{def:set_nonneg} with $|A^{\top}c|$, and denote the resulting set by $\Q^{\text{sign}}(c, \lambda)$. We then refine the space partition to guarantee that the set $\Q^{\mathrm{sign}}(c, \lambda)$ remains constant within each cell. Specifically, to achieve component-wise comparisons of absolute values against $\lambda$,  we partition $\Re^{r}\times \Re$ using the following $2n$ hyperplanes
\begin{align*}
H_i &= \left\{\left(c, \lambda\right)\in \Re^{r}\times \Re: \left( A^{\top} c \right)_i=\lambda \right\}, \ \ \forall i\in [n],\\
H_{i+n} &= \left\{\left(c, \lambda\right) \in \Re^{r}\times \Re: \left( A^{\top} c \right)_i= -\lambda \right\}, \ \ \forall i\in [n].
\end{align*}
The analysis of \Cref{prop:nonneg-standard-comonotone} extends directly and leads to  \Cref{prop:sign-inv-standard-comonotone}. 

\begin{proposition}
\label{prop:sign-inv-standard-comonotone}
Suppose that $\X$ is sign-invariant and standard \comonotone. Then 
\begin{enumerate}[(i)]
	\item There exists a collection of  $\mathcal{O}(n^{r})$ candidate supports for \eqref{eq:general-concave}, among which at least one support is optimal.
	\item Under \Cref{assume:supp},  problem \eqref{eq:general-concave} admits a polynomial-time algorithm with complexity $\O(n^{r} \cdot \textsf{T}_1)$.
\end{enumerate}
\end{proposition}

We illustrate further uses of the lifting idea in the subsequent sections.

\section{Complexity analysis beyond convexity and comonotonicity}\label{sec:extension}
In previous sections, we assume a finite convex objective function and a comonotone feasible region. In this section, we extend the established complexity results to two settings where these assumptions fail.

\subsection{Quasi-convex objective}
In some applications, such as the minimal cost-reliability ratio spanning tree problem \cite{chandrasekaran1981minimal}, the objective in \eqref{eq:general-concave} is a \emph{quasi-convex} rather than convex function. To be more precise, a function $g:\set{D}\to\R$ is called  \emph{quasi-convex}, where $\set{D}\subset\R^n$ is a convex open set, if the level set $\{ x\in\set{D}:g(x)\le t\}$ is convex for all $t\in\R$. An important source of quasi-convex functions arises in fractional programming \cite{goyal2013fptas, megiddo1978combinatorial, shigeno1995algorithm, hashizume1987approximation}. Specifically, the function $g(x)=\frac{g_1(x)}{g_2(x)}$ is quasi-convex when $g_1:\set{D}\to\R_+$ is convex and $g_2:\set{D}\to\R_{++}$ is concave. Moreover, composing a quasi-convex function with a univariate monotone nondecreasing function preserves quasi-convexity.  Similar to Proposition~\ref{prop:linear}, under additional regularity conditions, a quasi-convex maximization problem can be reduced to a nominal linear optimization problem.
\begin{proposition}\label{prop:quasi-convex2linear}
Suppose that $f:\set{D}\to\R$ is a upper semicontinuous and quasi-convex function, where $\set{D}\subseteq\R^r$ is open and convex, and the set $A\set{X}\defeq\{ Ax:x\in\set{X} \}\subseteq\set{D}$ is compact. Then there exists a vector $c\in\R^r$ such that every optimal solution to the linear counterpart~\eqref{eq:linear} is also optimal for the original problem~\eqref{eq:general-concave}. 
\end{proposition}
\begin{pf}
Since $f$ is upper semicontinuous and $A\set{X}$ is compact, the problem $\max\limits_{y\in A\set{X}}f(y)$ always admits an optimal solution $\bar y$. Let $f_{\max}=f(\bar y)$. Because $f$ is quasi-concave and upper semicontinuous, the strict sublevel set $\set{S}\defeq\{ y:f(y)< f_{\max} \}$ is an open convex set. Moreover, since $\bar y\notin \set{S}$, one can deduce from the renowned hyperplane separating theorem that there exists a $c\in\R^r$ such that 
\[ \set{S}\subseteq\{ y: c^\top(y-\bar y)\le0\}. \] Because $\set{S}$ is open, the inclusion can be strengthened to
\begin{equation}\label{eq:open-set-inclusion}
	\set{S}\subseteq\{ y: c^\top(y-\bar y)<0\}.
\end{equation}

Now let $y^*$ be any optimal solution to $\max\limits_{y\in A\set{X}}\;c^\top y$. Then $c^\top y^*\ge c^\top \bar y$, implying $y^*\notin \set{S}$ by~\eqref{eq:open-set-inclusion}.  Hence, by the definition of $\set{S}$, one has $\displaystyle y^*\in\mathop{\argmax}_{y\in A\set{X}}f(y)$. Finally, any optimal solution $x^*$ for  \eqref{eq:linear} leads to an optimal solution $Ax^*\in\argmax\limits_{y\in A\set{X}}f(y)$. 
\end{pf}

Because the convexity of the objective of \eqref{eq:general-concave} is only used in Proposition~\ref{prop:linear} to derive the results in Section~\ref{sec:poly} and \ref{sec:inv}, Proposition~\ref{prop:quasi-convex2linear} immediately implies the following remark.
\begin{remark}
The theoretical complexity bounds established in Section~\ref{sec:poly} and \ref{sec:inv} continue to hold for \eqref{eq:general-concave} with a quasi-convex objective $f$, provided that  the conditions in Proposition~\ref{prop:quasi-convex2linear} are satisfied.
\end{remark}

We emphasize that that the upper semicontinuity and compactness assumptions imposed in Proposition~\ref{prop:quasi-convex2linear} are necessary, which we illustrate in Example~\ref{ex:usc} and \ref{ex:compactness} below. 
Therefore, Proposition~\ref{prop:quasi-convex2linear} does not subsume Proposition~\ref{prop:linear}.
\begin{example}\label{ex:usc}
Let $n=r=1$, $\set{X}=\R$, $A=1$,  and $f(x)=\min\{|x|,1\}$. One can verify that the maximizers \eqref{eq:general-concave} are $(-\infty,-1]\cup[1,\infty)$. In contrast, for any nonzero $c\in\R$, the linear counterpart $\max\limits_{x\in\set{X}} c^\top x$ is unbounded and hence has no optimal solution. 
If $c=0$, then every $x\in\set{X}$ is optimal for the linear problem.
Thus, without compactness of $A\set{X}$, the conclusion of Proposition~\ref{prop:quasi-convex2linear} may fail. \hfill\Halmos
	\end{example}
	
	\begin{example}\label{ex:compactness}
Let $n=2$ and $A=I$. Define
\[ \set{X}=\bigg\{ (x_1,x_2):\, 0\le x_1\le \sqrt{1-x_2^2} \bigg\}\cup\bigg\{ (x_1,x_2):\,-1\le x_1,x_2\le 1 \bigg\}, \]
and 
\[ f(x_1,x_2)=\begin{cases}
		1&\text{if } x_2>1,\\
		0&\text{if }x=(0,1),\\
		-1&\text{otherwise}.
	\end{cases} \]
One can verify that $f$ is quasi-convex but not upper semicontinuous. Moreover, it admits a unique maximizer $\bar x=(0,1)$ over $\set{X}$. 
However, $\bar x$ is not an exposed point of $\set{X}$ (see Figure~\ref{fig:non-exposed-case}). As a result, no linear objective can single out $\bar x$. \hfill\Halmos

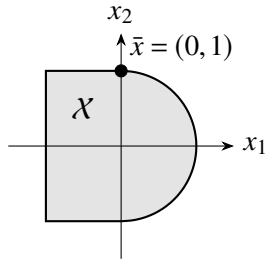
\begin{figure}[H]
	\begin{center}
		\begin{tikzpicture}%
			\tikzset{every path/.style={-}}
			\fill[black!15,opacity=0.7] (-1,-1) rectangle (0,1);
			
			\fill[black!15, opacity=0.7,domain=-90:90,smooth,variable=\t]
			plot ({cos(\t)},{sin(\t)}) -- (0,1) -- (0,-1);
			
			\draw[thick] (0,-1)--(-1,-1) -- (-1,1) -- (0,1);
			
			\draw[thick,domain=-90:90,smooth,variable=\t] plot ({cos(\t)},{sin(\t)});
			\draw[-Stealth] (-1.5,0) -- (1.5,0) node[right] {$x_1$};
			\draw[-Stealth] (0,-1.5) -- (0,1.5) node[above] {$x_2$};
			
			\filldraw[black] (0,1) circle (0.08) node[above right] {$\bar{x}=(0,1)$};
			\node[black] at (-0.5, 0.5) {$\set{X}$};
		\end{tikzpicture}
	\end{center}
	\caption{ The optimal solution cannot be exposed in $\set{X}$.}\label{fig:non-exposed-case}
\end{figure}
\end{example}

\subsection{Affine restriction of binary comonotone sets}
In this subsection, we relax the assumption that $\set{X}$ is a \comonotone set. Instead, we study the case where $\set{X}=\set{S}\cap\set{P}$ satisfying Assumption~\ref{assume:affine-restriction} below.
\begin{assumption}\label{assume:affine-restriction}
$\conv(\set{S}\cap\set{P})=\conv(\set{S})\cap\set{P}$, where $\set{S}\subseteq\{0,1\}^n$ is \comonotone and $\set{P}=\{x\in\R^n:\, Mx\le b\}$ for some $M\in\R^{m\times n}$ and $b\in\R^m$. 
\end{assumption}  Assumption~\ref{assume:affine-restriction} holds for example in either of the following cases: (i)~$\set{X}$ is a matroid and $\set{P}$ is a matroid polytope \cite{edmonds1970}, or (ii)~the affine constraints defining $\set{P}$ are \emph{facial} for $\set{X}$; see \cite[Section~3.1]{balas2018disjunctive} for a formal definition and further discussion.

The additional constraints imposed by $\set{P}$ can destroy the \comonotone structure of $\set{X}$. Proposition~\ref{prop:eliminating-linear-constraints} shows that their effect can be absorbed into a modified linear objective over $\set{S}$.

\begin{proposition}\label{prop:eliminating-linear-constraints}
Suppose $\set{X}=\set{S}\cap\set{P}$ satisfies Assumption~\ref{assume:affine-restriction}. Then for every $v\in\R^n$, there exists $\gamma\in\R_+^m$ such that $\displaystyle\bar x\in\mathop{\argmax}\limits_{x\in\set{X}} v^\top x$ if and only if 
\begin{equation}\label{eq:intersection-three-sets}
	\bar x\in\set{X}\cap \set{P}_=(\gamma)\cap\left[\mathop{\argmax}_{x\in\set{S}}\left(v-M^\top\gamma\right)^\top x\right],
\end{equation}
where $\set{P}_=(\gamma)\defeq\left\{ x\in\R^n:\,(Mx)_i=b_i\;\forall i\in[m] \text{ s.t. }\gamma_i\neq0 \right\}.$
\end{proposition}
\begin{pf}
First, suppose that $\bar x\in\mathop{\argmax}\limits_{x\in\set X}v^\top x$, and denote the optimal value by $\bar u$.  Then we also have that $\bar u =\max\limits_{x\in\conv(\set{X})}v^\top x$. By Assumption~\ref{assume:affine-restriction},  $\conv(\set{X})=\conv(\set S)\cap\set P$, which implies that
\begin{equation}\label{eq:convexified-problem}
	\begin{aligned}
			\bar u=\max\;\left\{ v^\top x: \; Mx\le b, x\in\conv(\set{S})\right\}.	
	\end{aligned}
\end{equation}
Let $\gamma\in\R_+^m$ be the optimal Lagrangian multipliers associated with the constraints $Mx\le b$ in \eqref{eq:convexified-problem}. Define the Lagrangian $\ell(x)=(v-M^\top \gamma)^\top x+b^\top\gamma$. It follows that 
\[ \bar u = \max_{x\in\conv(\set{S})}\ell(x)=\max_{x\in\set S}\ell(x), \]
implying $\bar x\in\mathop{\argmax}\limits_{x\in\set{S}}\ell(x)$. 
Moreover, complementary slackness implies that  $\bar x\in\set{P}_=(\gamma)$. Together with $\bar x\in\set{X}$, this yields~\eqref{eq:intersection-three-sets}.

We now prove the converse. Suppose that $\bar x$ satisfies \eqref{eq:intersection-three-sets}. Since $\bar x\in\set{P}_=(\gamma)$ implies $\gamma^\top(Mx-b)=0$, we have $\ell(\bar x)=v^\top \bar x$. Taking any $x\in\set X\subseteq\set{S}$, one can deduce from \eqref{eq:intersection-three-sets} that 
\[ \ell(\bar x)\ge \ell(x)=v^\top x-\gamma^\top(Mx-b)\ge v^\top x, \]
where the last inequality is due to $Mx\le b$ and $\gamma\ge0$. Because $\bar x\in\set{X}$ is feasible,  combining these relations proves that $\bar x$ is optimal for $\max\limits_{x\in\set{X}} v^\top x$.  
\end{pf}

Next, we consider a setting where in addition to Assumption~\ref{assume:affine-restriction}, we assume $\set{S}$ is standard comonotone. With this additional structure, Proposition~\ref{prop:eliminating-linear-constraints} enables us to reduce the analysis to linear optimization over $\set{S}$ and to derive the corresponding complexity results.  We present another application of Proposition~\ref{prop:eliminating-linear-constraints}  in Section~\ref{sec:disjoint-spca}.

As in the previous sections, we begin by  introducing a core set that captures the optimality conditions for cost vectors in a hyperplane arrangement cell to be defined later.
\begin{definition}
For any $c\in\R^r$, $\lambda\in\R$ and $\gamma\in\R^m$, define the set
\[ \Qaffine(c,\lambda,\gamma)\defeq \left\{x\in\R^n\left|\; \begin{aligned}
	x_{i} =0&  \text{ if } \left( A^\top c-M^\top \gamma \right)_i > \lambda\\
	x_{i} =1&  \text{ if } \left( A^\top c-M^\top \gamma\right)_i < \lambda
\end{aligned} \right.
\right\}.
\]
\end{definition}
We denote by $\textsf{LP}$ the worst-case time for both minimizing and maximizing $\displaystyle e^\top x$ over $\Qaffine(c,\lambda,\gamma)\cap\set{X}\cap\set{P}_=(\gamma)$.
\begin{lemma}\label{lem:opt-cond-aff-perm}
Suppose $\set{X}=\set{S}\cap\set{P}$ satisfies Assumption~\ref{assume:affine-restriction} and $\set{S}$ is standard comonotone. Then for all $c\in\R^r$, there exists $\gamma\in\R_+^m$ and $\lambda\in\R_+$ such that the following holds:
\begin{enumerate}[(1)]
	\item If  $\bar x\in\mathop{\argmax}\limits_{x\in\set{X}}c^\top Ax$, then $\bar x$ is either a minimizer or a maximizer of $e^\top x$ over $\Qaffine(c,\lambda,\gamma)\cap\set{X}\cap\set{P}_=(\gamma)$.
	\item Conversely, if $x^{\max}$ and $x^{\min}$ are any maximizer and minimizer of $e^\top x$ over $\Qaffine(c,\lambda,\gamma)\cap\set{X}\cap\set{P}_=(\gamma)$ respectively, then $\{x^{\max},x^{\min}\}$ contains an optimal solution to $\max\limits_{x\in\set{X}}c^\top Ax$.
\end{enumerate}
\end{lemma}
\begin{pf}		
Fix $c\in\R^r$ and let $\bar x\in\mathop{\argmax}\limits_{x\in\set{X}}c^\top Ax$. Applying Proposition~\ref{prop:eliminating-linear-constraints} to $v=A^\top c$, one can deduce that there exists $\gamma\in\R_+^m$ such that $\bar x$ is optimal for
\begin{equation}\label{eq:aux-lem-aff-res}
	\max_{x\in\set{S}}\left(A^\top c-M^\top\gamma\right)^\top x.
\end{equation}			
Moreover, since $\mathcal{S}$ is a binary standard comonotone set, \Cref{prop:nonneg-standard-comonotone} yields a threshold $\lambda\in\R$ such that $\bar x\in\Qaffine(c,\lambda,\gamma)$. Let $\mathcal{I}^==\{ i\in[n]:\,(A^\top c-M^\top\gamma)_i=\lambda \}$. Because for any $x\in \set{S}\cap \Qaffine(c,\lambda,\gamma)$, the coordinates $x_i$ are fixed for all $i\notin \set{I}^=$, the problem \eqref{eq:aux-lem-aff-res} is equivalent to 
\begin{equation*}
	\max_{x\in\set{S}\cap\Qaffine}\;\lambda\sum_{i\in\set{I}^=}x_i,
\end{equation*}	
which in turn  further amounts to 
\begin{equation}\label{eq:aux-lem-aff-res2}
	\max_{x\in\set{S}\cap\Qaffine}\;\lambda e^\top x.
\end{equation} Consequently, $\bar x$ is an optimal solution to \eqref{eq:aux-lem-aff-res2}. Depending on the sign of $\lambda$ and thanks to $\bar x\in\set{X}\cap\set{P}_=(\gamma)$, the first conclusion holds true.

To prove the second part, we assume WLOG that $\lambda\ge0$ since the other case can be proved similarly. Let $x^{\max}$ be a maximizer of $e^\top x$ over $\Qaffine\cap\set{X}\cap\set{P}_=(\gamma)$. Then $x^{\max}$ is also optimal for \eqref{eq:aux-lem-aff-res} by the equivalence above. Consequently, we deduce from Proposition~\ref{prop:eliminating-linear-constraints} and $x^{\max}\in\set{X}\cap\set{P}_=$ that $x^{\max}\in\mathop{\argmax}\limits_{x\in\set{X}}c^\top Ax$. 
\end{pf}

We now state the resulting complexity bound.
\begin{proposition}\label{prop:binary-almost-standard-comonotone}
Suppose $\set{X}=\set{S}\cap\set{P}$ satisfies Assumption~\ref{assume:affine-restriction} and $\set{S}$ is standard comonotone. Then problem~\eqref{eq:general-concave} can be solved in time $\mathcal{O}((n+m)^{r+m}\cdot(\textsf{T}_1+\textsf{LP}))$. 
\end{proposition}
\begin{pf}	
Consider the hyperplane arrangement $\mathcal{A}$ induced by 
\begin{align*}
	&H_i=\left\{ (c,\lambda,\gamma)\in\R^{r+m+1}:\,(A^\top c-M^\top\gamma)_i-\lambda=0 \right\}&\forall i\in[n]\\
	&H_j'=\left\{ (c,\lambda,\gamma)\in\R^{r+m+1}:\,\gamma_j=0 \right\}&\forall j\in[m].
\end{align*}
Then by construction,  for all tuples $(c,\lambda,\gamma)$ lying in the same hyperplane arrangement cell, the induced linear optimization problems
\begin{equation}\label{eq:either-min-or-max}
	\begin{aligned}
		\mathop{\min/\max}_{x\in\R^n}\;&e^\top x\\
		\text{s.t. }&x\in \Qaffine(c,\lambda,\gamma)\cap\set{X}\cap\set{P}_=(\gamma)
	\end{aligned}
\end{equation} share a common minimizer and a common maximizer. 

On the other hand, Proposition~\ref{prop:quasi-convex2linear} yields a cost vector $\bar c$ such that any optimal solution to 
\begin{equation*}
	\begin{aligned}
		\max\;& (A^\top \bar c)^\top x\\
		\text{s.t. }& x\in \set{X}=\set{S}\cap\set{P}
	\end{aligned}
\end{equation*}
is optimal for the original problem \eqref{eq:general-concave}. Combining this with Proposition~\ref{prop:eliminating-linear-constraints}  and Lemma~\ref{lem:opt-cond-aff-perm}, there exists $\bar\gamma\in\R^m$ and $\bar \lambda\in\R$ such that either the minimizer or the maximizer of \eqref{eq:either-min-or-max} associated with $(\bar c,\bar\lambda,\bar\gamma)$ is optimal for \eqref{eq:general-concave}.  Because $\set{A}$ gives rise to a partition of the parameter space $(c,\lambda,\gamma)$,   collecting the minimizers and maximizers  over all regions must include one optimal solution to \eqref{eq:general-concave}. Hence, the total number of candidate solutions is at most $2|\set{A}|$. The conclusion follows from \Cref{lem:HA-counting} that $|\set{A}|=\mathcal{O}((n+m)^{r+m})$.   
\end{pf}

\Cref{prop:binary-almost-standard-comonotone} shows that for almost binary standard comonotone feasible regions (i.e., with small $m$),  problem~\eqref{eq:general-concave} remains tractable.

\section{Applications}\label{sec:app}
The theoretical results developed in the previous sections are fairly general and  applicable to a broad range of problems, including the applications highlighted in Section~\ref{sec:intro}. Rather than presenting a long catalogue of examples, many of which follow directly once cast as instances of \eqref{eq:general-concave} (for example, matroid-constrained instances), we focus in this section on SPCA-related applications. This family offers a clean and representative setting for illustrating the implications of our general theory.
Beyond serving as an illustration, the analysis of these SPCA-related problems is of interest in its own right.

Throughout this section, we assume that the sample covariance matrix in question admits a rank-$r$ factorization $A^{\top}A\in\Re^{n\times n}$ with $A\in\Re^{r\times n}$. Note that $r$ is a fixed integer.
For each application, we either (i) match or improve the best-known computational complexity, or (ii) provide the first-known polynomial-time guarantee in fixed-rank settings. Table~\ref{tab:application} summarizes our results and compares them with the existing bounds.

\subsection{SPCA with a single component and its nonnegative variant}

Since the work of \cite{hotelling1933analysis}, principle component analysis (PCA) has been a widely-used tool for dimensionality reduction in statistics and machine learning, but its principal components typically involve all features, which can hinder interpretability and lead to unstable estimates.
SPCA addresses these issues by restricting the number of features used in each component \cite{jeffers1967two}. In contrast to PCA, which can be solved directly using eigenvalue decomposition, SPCA is \np-hard and even inapproximable in general \cite{magdon2017np}. 
However, in the fixed-rank setting, SPCA admits polynomial-time algorithms \cite{del2023sparse}.

This subsection studies SPCA with a single component and its nonnegative variant. The general SPCA problem is treated separately in \Cref{subsec:spca} because it requires a different analysis.
Formally, SPCA with a single component is defined as:
\begin{align}\label{eq:single_spca}
\max_{x\in \Re^{n}} \left\{\left\| A x\right\|_2^2: \|x\|^2_2= 1,\, \|x\|_0\le s \right\}, \tag{Single SPCA}
\end{align}
where $s\in [n]$.
It is evident that  \ref{eq:single_spca} fits within the framework of \eqref{eq:general-concave}, with the objective $f(Ax)=\| A x\|_2^2$. In this case, the rank of the function $f$ coincides with the rank $r$ of the covariance matrix $A^{\top}A$.
Moreover, the feasible set of \ref{eq:single_spca} is permutation- and sign-invariant with respect to $x$, which makes \Cref{prop:sign-inv-standard-comonotone} applicable. 

As shown in \cite{moghaddam2005spectral} (see also the proof of Lemma~\ref{lem:nnspca-subproblem} below), once the support of features is fixed, \ref{eq:single_spca} reduces to calculating the largest eigenvalue of an at most $s\times s$-sized principal submatrix of $A^{\top}A$. This reduction can be computed in $\textsf{T}_1=\mathcal{O}(s^2)$ time \cite{shishkin2019fast} and thus implies that \Cref{assume:supp} is satisfied.
Combining this result with \Cref{prop:sign-inv-standard-comonotone}, we obtain the following complexity bound for \ref{eq:single_spca}.
\begin{proposition}\label{prop:single_spca}
The following hold for \ref{eq:single_spca}:
\begin{enumerate}[(i)]
	\item There exists a collection of  $\mathcal{O}(n^{r})$ candidate supports for \ref{eq:single_spca}, among which at least one support is optimal.
	\item \ref{eq:single_spca} admits a polynomial-time algorithm with complexity $\O(n^{r} \cdot s^2)$.
\end{enumerate}
\end{proposition}

\Cref{prop:single_spca} matches the best-known bound of \ref{eq:single_spca}  in \cite{asteris2014sparse}.

\subsubsection{Nonnegative SPCA}  
In many applications, it is natural to impose nonnegativity constraints on the loadings, leading to \emph{nonnegative PCA} and its sparse variant. Two common modeling motivations are (i) to reflect  domain constraints where the latent direction is inherently nonnegative (e.g., intensities, concentrations, gene expression, metabolite abundances) \cite{allen2011sparse,sigg2008expectation}, and (ii) to avoid components that rely on positive--negative cancellations (i.e., contrast directions) \cite{zass2006nonnegative}. By incorporating this additional structure, nonnegativity can often improve interpretability \cite{asteris2014nonnegative} and reduce the estimation error of underlying statistical models; see \cite{montanari2015non} and the literature therein for more details.

Even without sparsity constraints, the nonnegative PCA problem remains \np-hard, as it includes the matrix copositivity testing problem as a special case, which is known to be \np-complete \cite{murty1987some}. We show below that when the covariance matrix $A^\top A$ has a fixed rank, the nonnegative SPCA problem is polynomially solvable. 
Formally, Nonnegative SPCA \eqref{eq:nn_spca} is defined as 
\begin{align}\label{eq:nn_spca}
\max_{x\in \Re^{n}} \left\{\left\| A x\right\|_2^2: x\ge0,\,\|x\|^2_2= 1, \|x\|_0\le s \right\}. \tag{NN-SPCA}
\end{align}

Next, we show how to solve the fixed-support subproblem for \ref{eq:nn_spca}. Given a support set $\supp(x)=S\subseteq[n]$, \ref{eq:nn_spca} reduces to
\begin{equation}\label{eq:ppca}
\begin{aligned}
	\max_{x\in\R^n}\, \left\{ \| A x\|_2^2: \, \|x\|^2_2= 1,\,  x_i>0, \, \forall i\in S,\, x_i=0 ,\,\forall i\notin S\right\}.
\end{aligned}
\end{equation}
Because  the feasible region of \eqref{eq:ppca} is not closed, its optimal solution may not exist.
\begin{lemma}\label{lem:nnspca-subproblem}There exists a polynomial algorithm that produces a feasible solution $\bar x$ to \ref{eq:nn_spca}. Moreover, whenever \eqref{eq:ppca} attains the optimum, the returned $\bar x$ is optimal for \eqref{eq:ppca}.
\end{lemma}  
\begin{pf}
Without loss of generality, we assume $S=[n]$; otherwise we can substitute out $x_i=0\;\forall i\notin S$ and apply the same argument to the corresponding principle submatrix of  $A^\top A$.  For any optimal solution $\bar x>0$ to \eqref{eq:ppca}, it must satisfy the KKT conditions for smooth optimization problems involving open sets (see \cite[Section~11.5]{mangasarian1994nonlinear}), implying that there exists $\bar \lambda\in\R$ such that 
\[ \nabla_x L(\bar x, \bar \lambda)=0, \]
where  $L(x,\lambda)=\|A x\|_2^2-\lambda\|x\|_2^2$. Thus, one has
\begin{equation}\label{eq:eigen-value}
	A^\top A \bar x=\bar \lambda \bar x, \text{ and } \bar x>0,
\end{equation}
which implies that $(\bar x, \bar \lambda)$ is a pair of eigenvector and eigenvalue of $A^{\top} A$. For each eigenvalue $\lambda$ of $A^\top A$, let $V(\lambda)\defeq\{ x\in\R^n:\, A^\top A x=\lambda  x\}$ denote the  associated eigenspace. Then by \eqref{eq:eigen-value}, solving \eqref{eq:ppca} boils down to finding the largest eigenvalue $\lambda$ such that $V(\lambda)\cap\R_{++}^n\neq\emptyset$. 

To test whether $V(\lambda)$ contains a strictly positive vector, consider the linear program
\begin{equation}\label{eq:LP-eigen}
	\begin{aligned}
		v(\lambda)\defeq\max_{t\in \Re,x\in \Re^n} \left\{t:\, x_i\ge t,\;\forall i\in [n], \,e^{\top}x= 1,\, x\in V(\lambda)\right\},
	\end{aligned}
\end{equation}
where the constraint $e^{\top}x= 1$ ensures a finite optimal value. For \eqref{eq:LP-eigen}, it is clear that $v(\lambda)>0$ if and only if $V(\lambda)\cap\R_{++}^n\neq\emptyset$. 
If $A^{\top} A$ admits an eigenvalue $\lambda$ satisfying $v(\lambda)>0$, let $\bar\lambda$ be the largest such eigenvalue, and let $x(\bar \lambda)$ be an optimal solution to \eqref{eq:LP-eigen} at $\bar \lambda$.
Then $\bar x \;\defeq\; {x(\bar\lambda)}/{\|x(\bar\lambda)\|_2}$ is feasible for \ref{eq:nn_spca}, and it is optimal for \eqref{eq:ppca} whenever \eqref{eq:ppca} attains an optimum. 
If $v(\lambda)\le0$ for all eigenvalues of $A^{\top}A$, then the optimal solution to \eqref{eq:ppca} does not exist. 
In this case, we set $\bar x$ as the first coordinate vector, which is trivially feasible for \ref{eq:nn_spca}.   

Since $A^\top A$ has at most $n$ distinct eigenvalues and each instance of \eqref{eq:LP-eigen} is a linear program, the procedure runs in polynomial time. 
\end{pf}

In parallel to \Cref{prop:single_spca}, we obtain the following complexity bound for \ref{eq:nn_spca}.
\begin{proposition}\label{prop:nnspca}
The following hold for \ref{eq:nn_spca}:
\begin{enumerate}[(i)]
	\item There exists a collection of  $\mathcal{O}(n^{r})$ candidate supports for \ref{eq:nn_spca}, among which at least one support is optimal.
	\item \ref{eq:nn_spca} admits a polynomial-time algorithm with complexity $\O\left(n^{r} \cdot \textsf{T}_1\right)$, where $\textsf{T}_1$ is the running time of the fixed-support routine in Lemma~\ref{lem:nnspca-subproblem}.
\end{enumerate}
\end{proposition}
\begin{pf}
Because the feasible region of \ref{eq:nn_spca} is nonnegative and standard comonotone, one can use \Cref{lem:nnspca-subproblem} in place of Assumption~\ref{assume:supp} and follow the same argument as in \Cref{prop:single_spca}. 
\end{pf}

As a corollary of \Cref{prop:nnspca}, we obtain the following result for nonnegative PCA (i.e., without the sparsity constraint). 
\begin{corollary}\label{cor:nnpca}
The nonnegative PCA problem
\begin{align}\label{eq:nn_pca}
	\max_{x\in \Re^{n}} \left\{\left\|A x\right\|_2^2: x\ge0,\,\|x\|_2^2= 1 \right\}. \tag{NN-PCA}
\end{align}
can be solved in polynomial time when the rank of $A$ is fixed.
\end{corollary}

The proof of \Cref{cor:nnpca} follows by setting $s=n$ in \ref{eq:nn_spca}.

\subsection{Variable selection for two-sample tests}
Two-Sample Tests (2ST) aim to determine whether two collections of samples are drawn from the same distribution, and they have found broad applications in bioinformatics, finance, healthcare, and machine learning \cite{gretton2012kernel,stein1945two}.
To enhance both statistical efficiency and interpretability,   \citet{wang2023variable} recently propose selecting a subset of informative variables to conduct 2ST based on the maximum mean discrepancy statistic. This leads to the optimization problem:
\begin{align}\label{eq:tst}
\max_{x\in \Re^{n}} \left\{\|A x\|_2^2 + a^{\top}x: \|x\|^2_2= 1, \|x\|_0\le s \right\}. \tag{2ST}
\end{align}

It is evident that \ref{eq:tst} and \ref{eq:single_spca} share the same feasible set. In fact, \ref{eq:single_spca} is a special case of \ref{eq:tst}: the only difference is the additional linear term  $a^{\top}x$ in the objective, which increases the 
the objective rank to $r+1$. More importantly, this term breaks the eigenvalue-based structure of \ref{eq:single_spca}. As a result,  existing complexity results (see, e.g., \cite{asteris2014sparse,del2023sparse}), which rely on eigenvalue properties, do not extend to \ref{eq:tst}.  Our general analysis framework can easily accommodate this setting. Observe that for any fixed support $S\subseteq [n]$, the corresponding subproblem for \ref{eq:tst} is a trust-region-type subproblem, which can be solved  by first computing an eigen-decomposition of $(A^\top A)_{S,S}$ and then solving a one-dimensional secular equation; see \cite{more1983computing} or Section~4.3 of  \cite{nocedal2006numerical} for details. Since the eigen-decomposition dominates the cost, the fixed-support subproblem can be solved in $\mathcal{O}(s^3)$ time.

As a direct application of \Cref{prop:sign-inv-standard-comonotone}, we obtain the first polynomial-time complexity bound for \ref{eq:tst} when $r$ is fixed.

\begin{proposition}
The following hold for \ref{eq:tst}: 
\begin{enumerate}[(i)]
	\item There exists a collection of  $\mathcal{O}(n^{r+1})$ candidate supports for \ref{eq:tst}, among which at least one support is optimal.
	\item \ref{eq:tst} admits a polynomial-time algorithm with complexity $\O(n^{r+1} \cdot s^3)$.
\end{enumerate}
\end{proposition}

\subsection{General SPCA} \label{subsec:spca}
In this subsection, we study the general SPCA problem with multiple principal components (see \cite{dey2023solving,vu2012minimax}):
\begin{align}\label{eq:spca}
\max_{U\in \Re^{n\times d}} \left\{\left\| A U\right\|_F^2: U^{\top}U= I, \|U\|_0\le s \right\}, \tag{SPCA}
\end{align}
where the \textit{row-sparsity} constraint $\|U\|_0\le s$ enforces the matrix $U$ to contain  at most $s$ nonzero rows.  \ref{eq:spca} is also more specifically  referred to as \emph{Row-Sparse PCA} in the literature \cite{dey2023solving,li2024beyond}. 

Since the general \ref{eq:spca} problem  does not admit a standard comonotone feasible region, the results in \Cref{sec:inv} cannot be invoked verbatim as in \ref{eq:single_spca}. Nevertheless, as shown in \citet{li2024beyond},  \ref{eq:spca} can be reformulated as the following convex maximization problem over a binary permutation-invariant set:
\begin{align}\label{eq:spca_binary}
\max_{x\in \{0,1\}^n} \bigg\{ \bigg\| \sum_{i\in [n]} x_i a_ia_i^{\top} \bigg\|_{(d)}: \sum_{i\in [n]} x_i\le s \bigg\}, 
\end{align}
where for each $i\in [n]$, $a_i\in \Re^r$ represents the $i$th column vector of $A$, and  $\|\cdot\|_{(d)}$ is a convex function given by the sum of the $d$ largest eigenvalues of its matrix argument.  \citet{shishkin2019fast} show that $\|\cdot\|_{(d)}$ for an $s\times s$ symmetric matrix can be computed in time $\textsf{T}_1=\O(ds^2)$. Note that \eqref{eq:spca_binary} naturally satisfies \Cref{assume:supp} because its feasible region is a binary set.
Moreover, since  $\sum_{i\in [n]} x_i a_ia_i^{\top}\in \Re^{r\times r}$ is symmetric and depends linearly on $x$, the objective in \eqref{eq:spca_binary} is  convex and has rank $(r^2+r)/2$.
Consequently, problem \eqref{eq:spca_binary} falls into the nonnegative and standard comonotone setting with rank $(r^2+r)/2$, and thus its complexity result immediately follows from \Cref{prop:nonneg-standard-comonotone}.

\begin{proposition}\label{prop:spca}
The following hold for \ref{eq:spca}:
\begin{enumerate}[(i)]
	\item There exists a collection of  $\mathcal{O}(n^{(r^2+r)/2})$ candidate supports for \ref{eq:spca}, among which at least one support is optimal.
	\item  \ref{eq:spca} admits a polynomial-time algorithm with complexity $\O(n^{(r^2+r)/2} \cdot ds^2)$.
\end{enumerate}
\end{proposition}

Compared with the support complexity $\O(n^{\min\{d, r\}(r^2+r)})$ derived in \cite{del2023sparse}, our new bound  $\O(n^{(r^2+r)/2})$ in \Cref{prop:spca} is independent of the number of principal components $d$ and cuts the exponent at least in half.

\subsection{Disjoint SPCA}\label{sec:disjoint-spca}
In this subsection, we study another variant of \ref{eq:spca} where the principal components are sparse and have disjoint supports. As originally proposed by \cite{asteris2015sparse}, we consider
\begin{align}\label{eq:ds_spca}
\max_{
	\begin{subarray}{l}
		Z\in \{0,1\}^{n\times d}:  \\
		Ze\le e, \;Z^{\top} e \le s
	\end{subarray} 
}\;\max_{U\in \Re^{n\times d}} \left\{\left\|A U\right\|_F^2: U^{\top}U= I_d,\, U_{ij} (1-Z_{ij})=0, \forall i\in [n],j\in [d] \right\}, \tag{Disjoint SPCA}
\end{align}
where each column of the binary matrix variable $Z\in \{0,1\}^{n\times d}$
encodes the support of each principal component, and the vector $s\in\mathbb{Z}_+^d$ specifies the sparsity budgets of the $d$ components. The constraints $Ze\le e$ model that the supports of components are disjoint. We note that \ref{eq:single_spca} is also a special case of \ref{eq:ds_spca} at $d=1$. We first reformulate \ref{eq:ds_spca} as a binary  optimization problem. %
\begin{lemma}
\ref{eq:ds_spca} is equivalent to
\begin{equation}\label{eq:ds_binary_spca}
	\begin{aligned}
		\max_{Z\in\{0,1\}^{n\times(d+1)}}\;\Bigg\{\sum_{j\in [d]} \lambda_{\max} \bigg(\sum_{i\in [n]} Z_{ij} a_ia_i^{\top}\bigg):&\, \sum_{j=1}^{d+1} Z_{ij}=1,\forall i\in[n],
		&\sum_{i=1}^n Z_{ij}\le s_j, \forall j\in[d]\Bigg\}.
	\end{aligned}
\end{equation}
\end{lemma}
\begin{pf}
In \ref{eq:ds_spca}, the matrix $U$ has disjoint column supports, which simplifies the orthonormal constraint $U^{\top}U=I_d$ to $U_j^{\top}U_j=1$ for all $j\in [d]$, where $U_j$ is denotes the $j$th column of $U$. Consequently,
the inner maximization problem over $U$ decomposes into  $d$ independent subproblems. According to \cite{li2025exact}, the optimal value of the $j$th subproblem is
\[\max_{U_j\in \Re^{n}} \left\{\left\|A^{\top} U_j\right\|_F^2: U_j^{\top}U_j= 1, U_{ij} (1-Z_{ij})=0, \forall i\in [n] \right\}=\lambda_{\max} \bigg(\sum_{i\in [n]} Z_{ij} a_ia_i^{\top}\bigg).\]
Plugging these values into \ref{eq:ds_spca} yields the objective of \eqref{eq:ds_binary_spca}. Finally,
we introduce a slack variable $Z_{i,d+1}\in\{0,1\}$ for each $i\in[n]$ to indicate whether the feature $i$ is used to construct principal components. This 
converts the original constraints $Ze\le e$ in \ref{eq:ds_spca} into $\sum_{j=1}^{d+1}Z_{ij}=1$ and completes the proof.
\end{pf}

Because the feasible set of \eqref{eq:ds_binary_spca}, denoted by $\set X$, is generally not comonotone, the complexity results from Section~\ref{sec:poly} and \ref{sec:inv} do not apply directly.  Fortunately, we can invoke Proposition~\ref{prop:eliminating-linear-constraints} to derive the optimality condition for \eqref{eq:ds_binary_spca}.
\begin{definition}
For any $V\in\R^{n\times(d+1)}$  and $\gamma\in\R^{d+1}$, define 
\[ \set{Q}^{DS}(V,\gamma)\defeq\left\{ Z\in\set X\left|\; \begin{aligned}
	&Z_{ij}=0,\;\forall i\in[n],\; j \notin  \argmax\{ V_{ij'}-\gamma_{j'}:\;j'\in[d+1]\}\\
	&\sum_{i=1}^n Z_{ij}=s_j \text{ if }\gamma_j>0, \;\forall j\in[d]
\end{aligned} \right. \right\}.  \]
\end{definition}
\begin{lemma}\label{lem:ds-spca-optimality}
For any $V\in\R^{n\times(d+1)}$, there exists a $\gamma\in\R^{d+1}_+$ with $\gamma_{d+1}=0$ such that  $\bar Z$ is an optimal solution to $\displaystyle \max_{Z\in\set{X}}\; \langle V,Z\rangle$ if and only if $\bar Z \in\set{Q}^{DS}(V,\gamma)$.
\end{lemma}
\begin{pf}
We write $\set X=\set S\cap \set P$, where \[\set{S} = \bigg\{ Z\in\{0,1 \}^{n\times(d+1)}: \sum_{j=1}^{d+1}Z_{ij}=1, \;\forall i\in[n] \bigg\}\] is the set of bases of a transversal matroid, and \[\set{P}=\bigg\{ Z\in\R^{n\times(d+1)}:  \sum_{i=1}^n Z_{ij}\le s_j, \;\forall j\in[d]\bigg\}\] is the corresponding transversal matroid polytope. 
The decomposition implies that $\conv(\set S\cap \set P)=\conv(\set S)\cap\set{P}$, and hence  Assumption~\ref{assume:affine-restriction} holds for \eqref{eq:ds_binary_spca}.  By Proposition~\ref{prop:eliminating-linear-constraints}, there exists $\gamma\in\R_+^d$ such that $
\bar Z\in \mathop{\argmax}\limits_{Z\in\set{X}} \langle V,Z\rangle$
if and only if (i) $\bar Z\in\set{X}$, (ii) $\sum_{i=1}^n \bar Z_{ij}=s_j$ for all $j\in[d]$ with $\gamma_j>0$, and (iii)
\begin{equation*}\label{eq:ds-spca-lifted-lp}
	\begin{aligned}
		\bar Z\in\mathop{\argmax}_{Z\in\set{S}}\;\langle V,Z \rangle -\gamma^\top(Z^\top e-s)= \mathop{\argmax}_{Z\in\set{S}}\;&\sum_{i=1}^n\left[  \sum_{j=1}^{d+1} (V_{ij}-\gamma_j) Z_{ij} \right].
	\end{aligned}
\end{equation*}
Since $\set S$ imposes a $(d+1)$-choose-one constraint for each row $i$, condition (iii) amounts to requiring that $\bar Z_{ij}$ can be nonzero only for indices $j$ attaining $\argmax_{j'\in[d+1]}\{V_{ij'}-\gamma_{j'}\}$. Together with (i) and (ii), we obtain $\bar Z\in \set{Q}^{DS}(V,\gamma)$.  
\end{pf}

We now state the complexity result for \ref{eq:ds_spca}.
\begin{proposition}\label{prop:disjoint-spca}
The following hold for \ref{eq:ds_spca}:
	\begin{enumerate}[(i)]
		\item There exists a collection of  $\mathcal{O}\left(( n(d+1)^2 )^{\frac{d(r^2+r+2)}{2}-1}\right)$ candidate supports for \ref{eq:ds_spca}, among which at least one support is optimal.
		\item 	\ref{eq:ds_spca} can be solved in  $\mathcal{O}\!\left(( n(d+1)^2 )^{\frac{d(r^2+r+2)}{2}-1}dn^2\right)$ time.
\end{enumerate}
\end{proposition}
\begin{pf}
Note that the objective of \eqref{eq:ds_binary_spca} is a sum of $d$ terms, where, for each $j\in [d]$, the $j$th term represents the largest eigenvalue of $\sum_{i\in [n]} Z_{ij} a_ia_i^{\top}$.
Analogous to the objective function of \eqref{eq:spca_binary}, each term is convex in $Z$ and has rank $(r^2+r)/2$. Consequently, the objective function in \eqref{eq:ds_binary_spca}  has rank $\tilde{r}\defeq d(r^2+r)/2$ and can be written in the form $f(\mathcal{L} Z)$, where $f$ is convex, and $\mathcal{L}:\R^{n\times (d+1)}\to \R^{\tilde{r}}$ is a linear operator. Let $\mathcal{L}^\top:\R^{\tilde{r}}\to\R^{n\times (d+1)}$ denote its adjoint. 

Let $\mathcal{A}$  be the hyperplane arrangement induced by the following hyperplanes
\begin{align*}
	&H_{ij\ell}=\left\{ (c,\gamma)\in \R^{\tilde{r}}\times\R^{d}:\;(\mathcal{L}^\top c)_{ij}+\gamma_j=(\mathcal{L}^\top c)_{i\ell}+\gamma_\ell \right\}, &\forall\,i\in[n],\;j\neq \ell \in[d+1],\\
	&H_j=\left\{(c,\gamma)\in\R^{\tilde{r}}\times\R^{d}:\;\gamma_j=0\right\},&\forall\,j\in[d],
\end{align*}
where we treat $\gamma_{d+1}=0$ and $(\mathcal{L}^{\top}c)_{i,d+1}=0\;\forall i\in[n]$ as fixed constants for notational convenience.
Then it follows from \Cref{lem:HA-counting} that
\begin{equation}\label{eq:ds-num-HA}
	|\A|=\mathcal{O}\left(\left( \frac{d(d+1)}{2}n+d \right)^{\tilde r+d-1}\right)=\mathcal{O}\left(\left( n(d+1)^2 \right)^{\frac{d(r^2+r+2)}{2}-1}\right).
\end{equation}

Combining \Cref{prop:linear} and \Cref{lem:ds-spca-optimality}, we obtain a certain $c\in\R^{\tilde{r}}$ and $\gamma\in\R^{\tilde{r}}$ such that $\set{Q}^{DS}(\mathcal{L}^\top c, \gamma)\neq\emptyset$ and any member of $\set{Q}^{DS}(\mathcal{L}^\top c, \gamma)$ is optimal for \eqref{eq:ds_binary_spca}. Furthermore, by the same reasoning as in the proof of \Cref{lem:HA}, the set $\set{Q}^{DS}(\mathcal{L}^\top c,\gamma)$ is constant within each region of $\mathcal{A}$. Therefore, selecting one representative point from $\set{Q}^{DS}(\mathcal{L}^\top c,\gamma)$ per region and collecting all such points yields a set of candidates that contains an optimal solution to \eqref{eq:ds_binary_spca}. 

To prove the complexity,  it remains to bound the cost per region. We first note that the feasibility problem over $\set{Q}^{DS}(\mathcal{L}^\top c, \gamma)$ is essentially a transshipment problem, which can be transformed into a maximum flow problem using standard techniques; see \cite[Section~11.6]{schrijver2003combinatorial}. Therefore, for each region, we can find a representative $Z\in\set{Q}^{DS}(\mathcal{L}^\top c, \gamma)$ in $O(dn^2)$ time using the Ford-Fulkerson Algorithm. Second, for each solution candidate $Z$,  evaluating the objective of \eqref{eq:ds_binary_spca} amounts to  computing the largest eigenvalues of $d$ symmetric matrices in $\R^{r\times r}$, which costs $\textsf{T}_1=\mathcal{O}(dr^2)$ in total. Together with \eqref{eq:ds-num-HA}, the overall running time is 
\[ |\A|\cdot (O(dn^2)+\mathcal{O}(dr^2))=\mathcal{O}\left(\left(  n(d+1)^2 \right)^{\frac{d(r^2+r+2)}{2}-1} dn^2\right). \]
This completes the proof. 
\end{pf}

Compared with \cite{del2023sparse}, \Cref{prop:disjoint-spca} provides a simpler analysis and improves the complexity bound $\mathcal{O}\!\big((dn)^{d^2(r^2+r)/2}\,(dnr^2+d^3n^5\log n)\big)$ reported therein.

\section{Conclusion}\label{sec:conclusion}
In this paper, we take a geometric view of the fixed-rank convex maximization problem \eqref{eq:general-concave}. We demonstrate that when the feasible region satisfies comonotonicity, a generalized symmetry introduced via conjugacy, the problem is polynomially solvable under mild conditions. This theoretical framework unifies and strengthens many existing complexity results in the literature.  In the standard comonotone case, we further sharpen the complexity bounds by partitioning an appropriate lifted parameter space.

Beyond these theoretical contributions, our analysis also yields a considerably simpler hyperplane-arrangement-based algorithm with  weaker dependence on the rank parameter than prior constructions in the literature. In particular, for standard comonotone sets, the dominant support search scales as $\mathcal{O}(n^{r})$ rather than $\mathcal{O}\!\left(n^{r(r+1)/2}\right)$. The reduced rank sensitivity and the simpler construction pave the way for practical implementations for solving reasonable low-rank approximations or convex relaxations of \eqref{eq:general-concave}. Building such implementations and further using low-rank surrogates to tackle the generic problem are left for future work.

\bibliographystyle{informs2014} %
\bibliography{Reference.bib} %

\end{document}